\documentclass[11pt,a4paper,twoside]{article}
\usepackage{color}
\usepackage{bm, amsmath, amssymb, amsthm} 
\usepackage{graphicx}

\def\calf{{\cal F}}
\def\<{\langle}
\def\>{\rangle}
\def\eps{\varepsilon}

\def\RR{\mathbb{R}}
\def\CC{\mathbb{C}}

\newcommand\const{\operatorname{const}}

\newcommand\tr{\operatorname{Tr}}
\newcommand\Div{\operatorname{div}}
\newcommand\id{\operatorname{id}}
\newcommand\im{\operatorname{im}}

\def\eq{\hspace*{-1.5mm}&=&\hspace*{-1.5mm}}
\def\plus{\hspace*{-1.5mm}&+&\hspace*{-1.5mm}}

\newcommand{\gv}{\mathrm{gv}}

\newtheorem{corollary}{Corollary}
\newtheorem{definition}{Definition}
\newtheorem{example}{Example}
\newtheorem{remark}{Remark}
\newtheorem{lemma}{Lemma}
\newtheorem{proposition}{Proposition}
\newtheorem{theorem}{Theorem}

\author{Vladimir Rovenski\footnote{Department of Mathematics, University of Haifa,
        e-mail: {\tt vrovenski@univ.haifa.ac.il}        }
        \ and \
        Pawe\l \  Walczak\footnote{Katedra Geometrii,
        Uniwersytet \L\'{o}dzki,
e-mail: {\tt pawel.walczak@wmii.uni.lodz.pl}
}
}

\title{Variations of the Godbillon--Vey invariant\\ of transversely parallelizable foliations}

\begin{document}

\date{}

\maketitle

\begin{abstract}
We consider a $(2q+1)$-dimensional smooth manifold $M$ equipped with a $(q+1)$-dimensional,
\textit{a~priori} non-integrable, distribution ${\cal D}$ and a $q$-vector field ${\bf T}=T_1\wedge\ldots\wedge T_q$, where $\{T_i\}$
are linearly independent vector fields transverse to~${\cal D}$. Using a $q$-form $\omega$ such that ${\cal D} = \ker\,\omega$
and $\omega({\bf T})=1$, we construct a $(2q+1)$-form analogous to that defining the Godbillon--Vey class of a $(q+1)$-dimensional foliation,
and show how does this form depend on $\omega$ and~${\bf T}$.
For a compatible Riemannian metric $g$ on~$M$, we express this $(2q+1)$-form in terms of ${\bf T}$ and extrinsic geometry of~${\cal D}$ and normal distribution ${\cal D}^\bot$.
We find Euler-Lagrange equations of associated functionals: for variable $(\omega,{\bf T})$ on $(M,g)$,
and for variable metric on $(M,{\cal D})$, when distributions/foliations and forms are defined outside~a~``singularity~set"
under additional assumption of convergence of certain integrals.
We show that for a harmonic distribution ${\cal D}^\bot$ such $(\omega,{\bf T})$ is critical,
characterize critical pairs when ${\cal D}$ is integrable and find sufficient conditions for critical pairs when variations are among foliations, calculate the index form and consider examples of critical foliations among twisted products, Reeb foliations and transversely holomorphic flows.


\end{abstract}

\section*{Introduction}

The \textit{Godbillon--Vey cohomology class} $\gv(\calf)$,
which occurs in algebraic topology, differential geometry and their applications~\cite{wpah}, 
was defined first for codimension-one foliations as a 3-cohomology class.
It proved to be one of the most interesting characteristic classes associated to a foliated manifold.
Its non-vanishing tells a lot about the dynamics of the foliation, e.g., implies the existence of resilient leaves.
The Godbillon--Vey class has been the subject of numerous publications by most prominent topologists interested in the foliation theory.
It is known that the Godbillon-Vey invariant is non-rigid,
and in \cite{rw-gv1} we studied the Godbillon-Vey invariant from the point of view of the variational calculus.
Then $\gv(\calf)$ was extended for foliations of codimension $q>1$, \cite{cc,gv},
and the paper generalizes the variational approach for such foliations.

If a codimension $q>1$ transversely oriented foliation $\calf$ of a closed manifold~$M$ is defined by the equation $\omega=0$
for some nowhere zero $q$-form $\omega$, then $\gv(\calf)$ is the de Rham cohomology class of the closed $(2q+1)$-form
$\eta\wedge (d\eta)^q$, where $\eta$ is a one-form obeying $d\omega =\omega\wedge\eta$.
The integrability condition for the tangent distribution
  $T\calf=\ker\omega:=\{X\in TM:\,\omega(X,\,\ldots)=0\}$,
implies the existence of such $\eta$, while $\gv(\calf)$ does not depend on the choice of $\omega$ and~$\eta$.
The~Godbillon--Vey class measures some sort of ``twisting" of the leaves, and it plays a crucial role in
topology and dynamics of foliations,
 see e.g. \cite{cc,glw,hu2002,tamura} and \cite[Problem~10]{hu2005}.
The~complex Godbillon--Vey class, defined for transversely holomorphic foliations of real codimension $2q$, is often referred as the \textit{Bott class}.
 If a smooth map $f:\bar M\to M$ is transverse to $\calf$ on $M$ then $\gv(f^*\calf)=f^*\gv(\calf)$, thus concordant foliations have the same Godbillon--Vey classes.
 When $\dim M = 2q+1$ we get a Godbillon--Vey number:
\begin{equation}\label{E-gv-invar0}
 \gv(\calf)=\int_M \eta\wedge (d\eta)^q.
\end{equation}
There exists one parameter family of foliations on $S^3$ with the Godbillon--Vey number taking all values in an interval
(for the particular Reeb foliation this number is zero), hence $\gv(\calf)$ is not a homotopy invariant.
As in the codimension one case, the Godbillon--Vey number \eqref{E-gv-invar0} is nonzero for various examples, and can even take on a continuum of values.
Variations of \eqref{E-gv-invar0} under deformations of $\calf$ have been studied in~\cite{as2015}.
 Our variational approach differs from one mentioned above. In our earlier work
\cite{rw-gv1} we defined a \textit{Godbillon--Vey type invariant} for a pair consisting of an arbitrary, \textit{a priori} non-integrable, plane field $\cal D$ and a transverse to it vector field $T$
on a Riemannian manifold $(M^3, g)$,
studied its dependence  on $\cal D$, $T$ and $g$,
found derivatives of the functional, characterized critical 2-dimensional foliations for different types of variations,
proved sufficient conditions for critical pairs when ${\cal D}$ varies over integrable plane fields (foliations),
found the index form of our variation problem, provided examples with Roussarie and Reeb foliations and twisted products.

Non-integrable distributions (subbundles of the tangent bundle $TM$) appear in many situations,
e.g. on contact manifolds and in sub-Riemannian geometry.
 A codimension $q$ distribution ${\cal D}$ can be defined by a \textit{locally decomposable $q$-form}
$\omega$,
that is $\omega=\omega_1\wedge\ldots\wedge\omega_q$ for some one-forms $\omega_i$ given in a neighborhood of a point $x\in M$.
Indeed, let $g$ be any Riemannian metric on $M$ and $T_i\ (1\le i\le q)$ a local basis of the distribution ${\cal D}^\bot$ normal to ${\cal D}$.  Then locally ${\cal D}=\ker\omega$, where $\omega=T_1^\flat\wedge\ldots\wedge T_q^\flat$.
The~``musical" isomorphisms $\sharp$ and $\flat$ ``lower'' and ``raise'' indices of rank one tensors on $(M,g)$.

A distribution is \textit{framed} if its ``normal bundle" $TM/{\cal D}$ is endowed with a trivialization.
In this paper, we consider a manifold $M^{2q+1}$ equipped with a $(q+1)$-dimensional distribution ${\cal D}$
and linearly independent vector fields $T_i\ (i\le q)$ transverse to ${\cal D}$.
Hence, ${\cal D}^\pitchfork={\rm span}(T_1,\ldots, T_q)$ is a smooth distribution
isomorphic to $TM/{\cal D}$.
Our~framed distribution ${\cal D}$ can be represented by a \textit{decomposable $q$-form} $\omega=\omega_1\wedge\ldots \wedge\omega_q$,
where $\omega_i$ are (not uniquely defined) one-forms.
Indeed, there exists a \textit{compatible Riemannian metric} $g=\<\,\cdot\,,\cdot\,\>$ on $M$,
i.e., the above vector fields $T_i\ (i\le q)$ are orthonormal and all are orthogonal to ${\cal D}$.
Denote by ${\rm Riem}(M,{\cal D},{\bf T})$ the space of all such metrics.
Given compatible metric $g$, set $\omega_i=T_i^\flat$.
Denote by ${\bf T}=T_1\wedge \ldots \wedge T_q$ a multivector on~$M$.

Operation $\iota_{\,\bf T}$ is defined on a differential $r$-form $\alpha$ with $r\ge q$~by
\[
 \iota_{\,\bf T}\,\alpha:=\iota_{\,T_q}\ldots\iota_{\,T_1}\,\alpha=\alpha(T_1,\ldots, T_q,\cdot\,,\ldots\cdot).
\]
For a decomposable $q$-form $\omega$ representing ${\cal D}$, one may assume the following normalization:
\begin{equation}\label{E-omega-T}
 \iota_{\,\bf T}\,\omega = 1.
\end{equation}
In fact, a pair $(\omega, {\bf T})$ or $({\cal D}, {\bf T})$, where ${\cal D}$ is represented by a decomposable $q$-form $\omega$ satisfying (\ref{E-omega-T}), is the main geometric structure considered here.
We build a $1$-form $\eta$ depending on $(\omega, {\bf T})$,
\begin{equation}\label{E-eta}
 \eta = \iota_{\,\bf T}\,d\omega,
\end{equation}
which is analogous to that defined in \cite{rw-gv1} for $q=1$, and study the functional
\begin{equation}\label{E-gv-invar}
 \gv: ( \omega, {\bf T}) \mapsto \int_M \eta\wedge (d\eta)^q.
\end{equation}
In~a~sense, our one-form $\eta$ arises from the best approximation of $(q+1)$-form $d\omega$ by the wedge-product of $\omega$ by a one-form, see Section~\ref{sec:1-1}.
We~provide variational formulas related to our construction and deduce Euler-Lagrange equations of \eqref{E-gv-invar} for variable pair $(\omega,{\bf T})$.

If $M$ is open, one may integrate over a relatively compact domain $G$ of $M$,
containing supports of variations of $(\omega,\,{\bf T})$ or a Riemannian metric.
Following ideas of \cite{br-w,pw1}, we consider singular foliations, distributions and forms,
that is those defined outside a ``singularity set" $\Sigma$,
\[
 \Sigma =
 \{\mbox{a~finite union of pairwise disjoint closed submanifolds of codim. $\ge k$}\}
\]
under an
assumption of convergence of improper integrals $\int_M \|\beta\|^p\,{\rm d}V_g$ for suitable $(\dim M-1)$-forms~$\beta$ defined on $M\smallsetminus\Sigma$ and some $p$ satisfying $(k-1)(p-1)\ge1$.

\smallskip

The~fundamental \textbf{question}~is:
  \textit{What are the best almost product structures on a manifold}?
Such pairs $(\omega,{\bf T})$ (of the above question) are proposed to be among critical points of \eqref{E-gv-invar}.
 We~show that $(\omega,{\bf T})$ is critical when the distribution ${\cal D}^\bot$ is harmonic (with respect to compatible metric),
characterize critical pairs when ${\cal D}$ is integrable
and find sufficient conditions for critical pairs when variations are among foliations,
calculate the index form and consider examples of critical foliations
among twisted products, Reeb foliations and transversely holomorphic flows.
We hope that the theory presented  here can be used to
deepen our knowledge of foliations as well as of topology of manifolds.

\section{Construction}
\label{sec:1-1}


The Lie derivative of differential forms along vector fields can be generalized to a Lie derivative of
differential forms along multivector fields, defined as graded commutator between the exterior derivative $d$ and the respective contraction operator: for a multivector ${\bf X}
= X_1\wedge\ldots X_r$ on $M$, see \cite{emr-1998,fo-2003},
\[
 {\cal L}_{\,\bf X} :=
  d\,\iota_{\,\bf X} -(-1)^r \iota_{\,\bf X}\,d ,
\]
where $\iota_{\,\bf X}\,\alpha := \alpha(X_1, \ldots, X_r, \ldots)$.
This leads to the relation
 $d{\cal L}_{\bf X} = (-1)^{r} {\cal L}_{\bf X}\,d$.

\begin{lemma}\label{L-iota-ab}
 Let ${\bf T}$ be a $q$-vector on $M$ and $\alpha,\beta$ differential forms.
 If $\deg\alpha+\deg\beta>\dim M+q-1$~then
\[
 \iota_{\,\bf T}\,\alpha\wedge\beta=(-1)^{q(\deg\alpha-1)}\alpha\wedge\iota_{\,\bf T}\,\beta.
\]
\end{lemma}

\begin{proof} For $q=1$ we have $\alpha\wedge\beta=0$ when $\deg\alpha+\deg\beta>\dim M$. Thus, $0=\iota_{\,T}\,(\alpha\wedge\beta)=\iota_{\,T}\,\alpha\wedge\beta+(-1)^{\deg\alpha}\alpha\wedge\iota_{\,T}\,\beta$.
Then setting  $\widetilde{\bf T}={\bf T}\wedge T_q$ and using  induction we get
\begin{eqnarray*}
 &&\hskip-5mm \iota_{\,\widetilde{\bf T}}\,\alpha\wedge\beta
 = \iota_{\,T_q}\,\iota_{\,\bf T}\,\alpha\wedge\beta
 = (-1)^{\deg\alpha-q}\,\iota_{\,\bf T}\,\alpha\wedge\iota_{\,T_q}\,\beta\\
 &&  \overset{\rm ind}= (-1)^{\deg\alpha-q}\cdot (-1)^{(q-1)(\deg\alpha-1)}\alpha\wedge\iota_{\,\bf T}\,\iota_{\,T_q}\,\beta\\
 &&  = (-1)^{q(\deg\alpha-1)}\alpha\wedge\iota_{\,T_q}\,\iota_{\,\bf T}\,\beta
 = (-1)^{q(\deg\alpha-1)}\alpha\wedge\iota_{\,\widetilde{\bf T}}\,\beta,
\end{eqnarray*}
proving the claim.
\end{proof}

Given $g\in{\rm Riem}(M,{\cal D},{\bf T})$, consider in the space $\Lambda^1(M)$ of one-forms on $M$ the subspace
 $\omega^\perp=\{\theta\in\Lambda^1(M):\ \theta(T_i)=0,\ 1\le i\le q\}$.
Consider also, in the space $\Lambda^{q+1}(M)$ of $(q+1)$-forms on $M$ the subspace
$(\RR\cdot\omega)\wedge\omega^\perp$
of all $(q+1)$-forms $\omega\wedge\theta$, $\theta$ being a one-form of $\omega^\perp$.
 Now, project $d\omega$ orthogonally onto the subspace $(\RR\cdot\omega)\wedge\omega^\perp$.
The~projection $(d\omega)^\bot$ has the form $\omega\wedge\eta$ with $\eta$ belonging to $\omega^\perp$. Such $\eta$ is unique.

\begin{proposition}\label{L-eta}
 The one-form $\eta$ does not depend on a compatible metric $g$, and is given by \eqref{E-eta}, or equivalently, $\eta = (-1)^{q-1} {\cal L}_{\,\bf T}\,\omega$.
\end{proposition}

\begin{proof}
The $g$-orthogonality $d\omega-\omega\wedge\eta\,\perp\,(\RR\,\omega)\wedge\omega^\perp$ means
 $(d\omega -\omega\wedge\eta)({\bf T},\,X)=0$ for any $X\in{\cal D}$;
thus, $\iota_{\,X}(\iota_{\,{\bf T}}\,d\omega-\iota_{\,{\bf T}}(\omega\wedge\eta))=0$  for any $X\in TM$.
Using $\eta(T_i)=0$
and \eqref{E-omega-T}
we get~\eqref{E-eta}.
\end{proof}

The $(2q+1)$-form $\eta\wedge(d\eta)^q$
represents the \textit{Godbillon--Vey type invariant} \eqref{E-omega-T}
of a pair $(\omega,{\bf T})$.

 Since all the vectors  uniquely decompose into ${\cal D}^\bot$- and ${\cal D}$- components,
there are 3 special cases for
another pair of the same sort, $(\tilde\omega,\widetilde{\bf T})$, satisfying~\eqref{E-omega-T}:

\smallskip
(i)~$\widetilde{\bf T}$ is parallel to ${\bf T}$ and $\tilde\omega$ is parallel to $\omega$,

\smallskip
(ii)~$\widetilde T_i-T_i$ belongs to ${\cal D}$ (hence $\iota_{\,\widetilde{\bf T}}\,\omega=1$)
and $\tilde\omega=\omega$,

\smallskip
(iii)~$\widetilde T_i=T_i$ and $\tilde\omega = \omega +\mu$ for some $q$-form $\mu$ such that $\iota_{\,\bf T}\,\mu=0$.

\smallskip\noindent
Notice that the distribution ${\cal D}$ is preserved in cases (i) and~(ii).


\begin{proposition}
The number $\gv(\omega,{\bf T})$ does not change when we modify $(\omega,{\bf T})$ as in case \rm{(i)},
that is $\widetilde T_i=C_i^j T_j$ and $\tilde\omega=\det C^{-1}\omega$ for some $C:M\to GL(q,\RR)$,
such that $\det C$ is constant on ${\cal D}^\bot$-curves.
\end{proposition}

\begin{proof}
Denote  a function $c$ on $M$ by $c:=\det C=\eps_{j_1,\ldots, j_q}C_1^{j_1}\cdot\ldots\cdot C_q^{j_q}$ .
 In this case, $(\tilde\omega,\widetilde{\bf T})$ obeys~\eqref{E-omega-T}:
\begin{eqnarray*}
 \iota_{\,\widetilde{\bf T}}\,\tilde\omega
 \eq C_1^{j_1}\cdot\ldots\cdot C_q^{j_q}\tilde\omega(T_{j_1},\ldots,T_{j_q})\\
 \eq c^{-1}\eps_{j_1,\ldots, j_q}C_1^{j_1}\cdot\ldots\cdot C_q^{j_q}\,\omega(T_{1},\ldots,T_{q})=\iota_{\,\bf T}\,\omega =1.
\end{eqnarray*}
 By Proposition~\ref{L-eta}, using conditions $T_i(c)=0$, we find
\begin{eqnarray*}
  \tilde\eta \eq (-1)^{q-1}{\cal L}_{\widetilde{\bf T}}\,( c^{-1}\,\omega)
  = c^{-1}\iota_{\widetilde{\bf T}}\,( c\,d( c^{-1})\wedge\omega + d\omega)\\
  \eq \eta - c^{-1}\iota_{\,\bf T}( d c\wedge\omega)  = \eta-d(\log c),
\end{eqnarray*}
thus, $\eta\wedge (d\eta)^q$ changes by the closed form $d(\log c)\wedge (d\eta)^q$ when $T_i(c)=0$. This implies the claim.
\end{proof}

\begin{definition}\label{D-t-umb}\rm
Let $\nabla$ be the Levi-Civita connection of $g\in{\rm Riem}(M,{\cal D},{\bf T})$.
The (non-symmetric) second fundamental form $h:{\cal D}\times{\cal D}\to{\cal D}^\bot$ of
${\cal D}$ (and similarly $h^\bot$ for the normal distribution ${\cal D}^\bot$) is
\begin{equation*}
 h_{X,Y} = (\nabla_X Y)^\bot\quad (X,Y\in {\mathfrak X}_{\cal D}).
\end{equation*}
A distribution ${\cal D}$ is called \textit{totally geodesic} if ${\rm Sym}(h)=0$ (the symmetrization of $h$), \textit{harmonic} if $H=0$,
and \textit{totally umbilical} if ${\rm Sym}(h)=g_{\,|{\cal D}}{\cdot}H$.
\end{definition}

 There exist many examples of distributions and foliations with such properties.
Assume that the mean curvature vector $H^\bot := \sum\nolimits_{\,i}(\nabla_{T_i}\,T_i)^\top$ of
${\cal D}^\bot$,
is nonzero on an open non-empty set $U\subset M$.

\begin{lemma}
\label{L-02a}
We have
\begin{eqnarray}\label{eq1:pw2013}
 && \eta = (-1)^{q-1}(H^\bot)^\flat,\\
\label{eq1:pw2013b}
 && (-1)^{q-1} d\eta(X,Y) = \<\nabla_X H^\bot, Y\> - \<\nabla_Y H^\bot, X\>.
\end{eqnarray}
\end{lemma}

\begin{proof} Using formula for the exterior derivative of a $q$-form,
we get for $X\in{\cal D}$ and ${\bf T}=T_1\wedge \ldots \wedge T_q$,
\begin{eqnarray*}
 && d\omega(X,{\bf T}) = X(\iota_{\,\bf T}\,\omega)
 +\sum\nolimits_{i} T_i(\omega(X,\widehat{\bf T}_i)\\
 && +\sum\nolimits_{i<j} (-1)^{i+j}\omega([T_i,T_j],X,\widehat{\bf T}_{i,j})
 +\sum\nolimits_{i} (-1)^{i}\omega([X,T_i],\widehat{\bf T}_i) \\
 && \sum\nolimits_{i} (-1)^{i}\<[X,T_i],T_i\>\,\omega(T_i,\widehat{\bf T}_i)
 =\sum\nolimits_{i} (-1)^{i}\<[X,T_i],T_i\>\,\omega(T_i,\widehat{\bf T}_i)\\
 && =\sum\nolimits_{i} \<\nabla_X\,{T_i}-\nabla_{T_i}\,X,T_i\>
 = -\<X,\sum\nolimits_{i}\nabla_{T_i}\,T_i\> =-\<X,H^\bot\>,
\end{eqnarray*}
where
$\widehat{\bf T}_i=T_1\wedge\ldots\wedge \widehat T_i\wedge\ldots\wedge T_q$,
and as usual the hat over a symbol denotes its omission.
For $X\in{\cal D}^\bot$ this is obvious.
Thus, \eqref{eq1:pw2013} follows.
Similar calculation for $d\eta(X,Y)$
and \eqref{eq1:pw2013}
yield~\eqref{eq1:pw2013b}.
\end{proof}

The unit vector $N=H^\bot/\|H^\bot\|$, called the \textit{principal normal} of the distribution ${\cal D}^\bot$, and the
\textit{binormal distribution} ${\cal B}={\cal D}\cap N^\bot$ are defined on $U$,
and by Lemma~\ref{L-02a}, we have
\[
 \eta=(-1)^{q-1} \|H^\bot\|\cdot N^\flat.
\]

\begin{proposition}\label{P-Hbot-0}
 If ${\cal D}=\ker\omega$ and $\gv(\omega,{\bf T})\ne0$ then there is no compatible metric $g\in {\rm Riem}(M,{\cal D},{\bf T})$
 such that the distribution ${\cal D}^{\,\bot}={\rm span}({\bf T})$ is harmonic with respect to $g$.
\end{proposition}

\begin{example}[Case $q=1$]\rm
Let $M^{3}$ be equipped with a plane field ${\cal D}=\ker\omega$ (for a one-form $\omega$), and a~vector field $T$
such that $\omega(T)=1$. Then $\eta=\iota_{\,T}\,d\omega$. For a compatible Riemannian metric $g$ on $M^3$,
the unit normal $N$, the binormal $B$ and the \textit{torsion} $\tau$ of $T$-curves are defined on an open subset $U$,
 where the \textit{curvature} $k$ of $T$-curves is nonzero. Then \eqref{eq1:pw2013} reads $\eta=k\,N^\flat$.
Let $T$ be a geodesic vector field on $M^3\setminus\Sigma$ with a Riemannian metric $g$, and $\omega=T^\sharp$.
Then $\eta$ vanishes, hence $\eta\wedge d\eta=0$.
\end{example}

\begin{example}\rm
Assume that $\dim M=2n+1\ge 2q+1$. Let a $q$-form $\omega$, a $q$-vector ${\bf T}$ and a one-form $\eta$
be as above.
Then the following {\it Godbillon--Vey type invariants} are well-defined, see \cite{fh,rw-gv1} for $q=1$:
\begin{equation*}
 \gv_{\bf s}(\omega,T) = \int_M \eta\wedge(d\eta)^{s_0}\wedge(d\omega_1)^{s_1}\wedge\ldots\wedge(d\omega_q)^{s_q},
\end{equation*}
where ${\bf s}=(s_0,\ldots,s_q),\ |{\bf s}|=s_0+\ldots+s_q=n$.
For $s_0=n=q$, we get the functional \eqref{E-gv-invar}.
If ${\mathcal D}=\ker\omega$ is integrable, i.e., $d\omega=\omega\wedge\eta$, which applying $d$ yields $\omega\wedge d\eta =0$;
then, since \eqref{E-deta-q} and $\omega=\omega_1\wedge\ldots\wedge\omega_q$ hold,
we have $\gv_{\bf s}(\omega,{\bf T}) = 0$ for all $s_0\ge q$.
To illustrate the above, consider
a 3-contact distribution ${\mathcal D}=\ker\omega$ on $M^{4n+3}$
with the Reeb field ${\bf T}=T_1\wedge T_2\wedge T_3$.
Then $\omega({\bf T})=1$ and $\eta:=\iota_{\,{\bf T}}\,d\omega$ vanishes, see \cite{b2010}, hence $\gv_{\bf s}(\omega,{\bf T})=0$.
The above Godbillon--Vey type invariants can be also applied to
globally framed $f$-structures and almost para-$\phi$-structures with complemented frames.
\end{example}

\section{Tautness}

Let ${\cal D}$ be a codimension $q$ distribution on $(M,g)$,
and ${\cal D}^{\,\pitchfork}$ a transverse distribution with a local orthonormal frame $\{T_i\},\ 1\le i\le q$.
Assume that ${\cal D}^{\,\pitchfork}$ is oriented and let $\Omega^{\,\pitchfork}$ be its volume form:
\begin{equation*}
 \Omega^{\,\pitchfork}(X_1, \ldots , X_{q}) = \det [\, \< X_i, T_j \>,\ i, j = 1, \ldots , q\,]
\end{equation*}
for all vector fields $X_1, \ldots , X_{q}$ on $M$.
Let $H^{\,\pitchfork}$ be the mean curvature vector field of ${\cal D}^{\,\pitchfork}$.

\begin{proposition}
For any vector field $Z$ on a Riemannian manifold $(M,g)$ one has
\begin{equation}\label{eq:rum}
 d\Omega^{\,\pitchfork}(Z, T_1, \ldots , T_{q}) = - \<Z, H^{\,\pitchfork}\>.
\end{equation}
\end{proposition}

For example, if ${\cal D}^{\,\pitchfork}$ is harmonic, i.e., $H^{\,\pitchfork}=0$,
then $\Omega^{\,\pitchfork}$ is ${\cal D}^{\,\pitchfork}$-{\it closed},
i.e.,
$d\Omega^{\,\pitchfork}(\,\cdot\,, X_1, \ldots, X_{q}) = 0$
for all $X_1, \ldots , X_{q}$ tangent to ${\cal D}^{\,\pitchfork}$.

 Denote by $D^k$ the
 linear space of (global) $k$-forms on $M$ equipped with the C$^\infty$-topology.
Recall that ~$k$-\textit{currents} are continuous (with respect to the weak $*$-topology) functionals
on the Freshet space $D^k$,  that is the space $D_k= (D^{k})^\star$ of $k$-{currents} is dual of $D^k$.
The space $D = \bigoplus_k D_k$ of all currents
 can be equipped with a linear \textit{boundary operator} $\partial:D\to D$, the adjoint to the
 exterior differential $d$:
 \begin{equation*}
 \partial : D_k\to D_{k-1}\ \text{and}\ \partial z(\Omega) = z (d\Omega)
 \ \text{for all}\ z\in D_k\ \text{and}\ \Omega\in D^{k-1}.
 \end{equation*}
 Certainly, $d^2 = 0$ implies $\partial^2 = 0$, therefore, one can consider
 the spaces $Z_k = \ker\partial\subset D_k$ of $k$-{\it cycles} and
 $B_k = \im\partial\in D_k$ of $k$-{\it boundaries}, and the corresponding
 homologies $H_k(M) = Z_k/B_k$.

Any $q$-vector $v = v_1\wedge\ldots\wedge v_q, \  v_i\in T_xM, \ x\in M$
defines a {\it Dirac current}~$z_v$:
\begin{equation*}
 z_v (\Omega) = \Omega (v_1\ldots , v_q), \quad \Omega\in D^q.
\end{equation*}
Consider the closed convex cone $C^{\,\pitchfork}\subset D_q$ generated by all Dirac currents $z_v$,
where $v = v_1\wedge\ldots\wedge v_{q}$ and $(v_1,\ldots , v_{q})$ is
a positive oriented frame of ${\cal D}^{\,\pitchfork}_x$ at a point $x\in M$.
 If $M$ is compact then the cone $C^{\,\pitchfork}$
 of ${\cal D}^{\,\pitchfork}$-currents
 has compact base.
(By a {\it base} of a cone $C$ contained in a topological vector space $V$ we mean
 the set $l^{-1}(1)$, where $l:V\to\RR$ is a continuous linear functional positive on $C\setminus\{0\}$).
 Let~$B^{\,\pitchfork}$ be the closed linear subspace of $D_q$ generated by the boundaries of all Dirac currents
 $z_v\in D_{q+1}$, where $v=w\wedge v_1\wedge\ldots\wedge v_{q}$ with $w\in T_xM, v_i\in {\cal D}^{\,\pitchfork}_x$ and $x\in M$.

\begin{definition}\label{defi:taut}\rm
 A pair $({\cal D},{\cal D}^{\,\pitchfork})$ of transverse distributions is called

 (1) {\it geometrically taut} if there exists a Riemannian metric $g$ on $M$
 for which the distribution ${\cal D}^{\,\pitchfork}$ becomes harmonic and orthogonal to ${\cal D}$ on $(M, g)$;

 (2) {\it topologically taut}
 if
 $C^{\,\pitchfork}$
 intersects trivially the smallest closed linear subspace $P^{\,\pitchfork}$ of ${D}_q$ containing $B^{\,\pitchfork}$ and all
 Dirac currents $z_v$, where $v=w\wedge v_1\wedge\ldots\wedge v_{q-1}$
 with $w\in{\cal D}_x$, $v_i\in T_xM$ and~$x\in M$.
\end{definition}

For integrable distributions $({\cal D},{\cal D}^{\,\pitchfork})$ (i.e., a pair of transverse foliations), the above $C^{\,\pitchfork},B^{\,\pitchfork}$ and the two kinds of tautness were introduced and studied in \cite{wa2013}.
Our goal here is to show that the two types of tautness in Definition~\ref{defi:taut} are equivalent when ${\cal D}^{\,\pitchfork}$ is integrable. To this end,
one has to apply the Sullivan's \textit{purification} of differential forms, see \cite{su2}.
If ${\cal D}^{\,\pitchfork}$ is not integrable then, unfortunately,  purification does not enjoy properties needed to prove equivalence of topological and geometrical tautness of the pair $({\cal D},{\cal D}^{\,\pitchfork})$.

\begin{theorem}\label{theo:taut}
 A pair $({\cal D},{\cal D}^{\,\pitchfork})$ with integrable $q$-dimensional distribution ${\cal D}^{\,\pitchfork}$ on a closed manifold $M^{2q+1}$ is geometrically taut if and only if it is topologically taut.
\end{theorem}

\proof
This is similar to proof in \cite[Section~3]{wa2013} when also ${\cal D}$ is integrable.

$\Rightarrow)$\ Let the pair $({\cal D},{\cal D}^{\,\pitchfork})$ be geometrically taut and $g$ a Riemannian metric making
${\cal D}^{\,\pitchfork}$ harmonic and orthogonal to ${\cal D}$.
The volume form $\Omega^{\,\pitchfork}$ of ${\cal D}^{\,\pitchfork}$ on $(M, g)$ is ${\cal D}^{\,\pitchfork}$-closed;
it is
positive on $C^{\,\pitchfork}\setminus\{0\}$ and equal identically to zero on $P^{\,\pitchfork}$.
Therefore, $C^{\,\pitchfork}\cap P^{\,\pitchfork}=\{0\}$ and $({\cal D},{\cal D}^{\,\pitchfork})$ is topologically taut.

$\Leftarrow)$\ Assume now that the pair $({\cal D},{\cal D}^{\,\pitchfork})$ is topologically taut.
Since, as we mentioned before, the cone $C^{\,\pitchfork}$ has a compact base $B$,
the Hahn--Banach Theorem implies the existence of a continuous linear functional $\lambda: D_q\to\RR$ such that
$\lambda = 0$ on $P^{\,\pitchfork}$ and $\lambda = 1$ on $B$.
 The Schwarz's Theorem says that $D_q$ is also dual to $D^q$, i.e.,
 each continuous linear functional on $D^q$ comes from evaluation the currents in $D^q$ on some fixed $q$-form $\Omega$.
 Hence, $\lambda$ represents a $q$-form $\Omega$: $z(\Omega) = \lambda(z)$ for any $z\in D_{q}=(D^q)^*$.
Since $\Omega$ is positive on $C^{\,\pitchfork}$, there exists a Riemannian metric $g$ on $M$ for which $\Omega$ is the volume form
of ${\cal D}^{\,\pitchfork}$.
Since $\Omega$ vanishes on $B^{\,\pitchfork}$, the distribution ${\cal D}^{\,\pitchfork}$ on $(M, g)$ is harmonic.
Since ${\cal D}^{\,\pitchfork}$ belongs to the kernel of $\Omega$,
we get ${\cal D}^{\,\pitchfork}\subset{\cal D}^\bot$.
Comparing dimensions one has ${\cal D}^{\,\pitchfork} = {\cal D}^\bot$.
By \eqref{eq:rum}, $H^\bot=0$.
Thus, $({\cal D},{\cal D}^{\,\pitchfork})$ is geometrically taut.
\qed

\smallskip

By Theorem~\ref{theo:taut} and Proposition~\ref{P-Hbot-0},
$\gv(\omega, {\bf T})$ is an obstruction for topological tautness of $({\cal D},\, {\cal D}^{\,\pitchfork})$.

\begin{corollary}\label{cor:taut}
Let ${\cal D}=\ker\omega$ be a codimension $q$ distribution on $M^{2q+1}$ defined by a $q$-form $\omega$,
and ${\cal D}^{\,\pitchfork}={\rm span}(T_1,\ldots, T_q)$ an integrable distribution
spanned by $q$ linearly independent vector fields $T_i$ transverse to ${\cal D}$.
If $\gv(\omega, T_1\wedge \ldots \wedge T_q)\ne 0$,
then
$({\cal D},\, {\cal D}^{\,\pitchfork})$ is niether topologically nor geometrically~taut.
\end{corollary}


\begin{example}[Geodesible vector fields on 3-manifolds]\rm
Recall the Sullivan's \cite{su2} characterization of geodesic fields, see also the survey in \cite{gl}:
 {Let $T$ be a
 nonsingular vector field on a smooth manifold $M$.
Then, there is a Riemannian metric making the orbits of $T$ geodesics if and only if no nonzero foliation cycle
for $T$ can be arbitrarily well approximated by the boundary of a 2-chain tangent to $N$}.
 According to Definition~\ref{defi:taut}, a~pair $({\cal D}=\ker\omega,T)$,
 where
 $\omega$ is a one-form on $M^3$ such that $\omega(T)=1$,
  is

 (1) {\it geometrically taut} if $T$ is geodesic and orthogonal to
 $\ker\omega$ for some Riemannian metric $g$ on~$M$;

 (2) {\it topologically taut} if
 $C^{\,\pitchfork}$
 intersects trivially the smallest
 linear subspace $P^{\,\pitchfork}$ containing $B^{\,\pitchfork}$ and
 all Dirac currents $z_{\,v}$ with $v\in{\cal D}_x$ and $x\in M$.
 The cone $C^{\,\pitchfork}$ is generated by Dirac current $z_{\,T}$,
and $B^{\,\pitchfork}$ is the {closed linear subspace} generated by boundaries
of Dirac currents $z_v\in D_{2}$, where $v=w\wedge T$.
By Theorem~\ref{theo:taut},
a pair $({\cal D}, T)$ on a closed manifold $M^{3}$ is geometrically taut if and only if it is  topologically taut.
 By Corollary~\ref{cor:taut},
if $\gv(\omega, {T})\ne 0$, then the pair $({\cal D},\, T)$ is not taut.
\end{example}

\section{Variations of distributions}
\label{sec:1-2}

The Stokes Theorem states that
 $\int_{M} d\beta = \int_{\partial M} \beta$,
when $\beta$ is a $(\dim M -1)$-form on $M$. Thus, $\int_{M} d\beta =0$, when $M$ is closed;
this is also true if $M$ is open and $\beta$ is supported in a relatively compact domain $G$.
 The Stokes Theorem on a closed Riemannian manifold $(M,g)$ with the volume form ${\rm d}V_g$ and $\beta=X^\flat$ yields the Divergence Theorem,
 $\int_{M} \Div X \,{\rm d}V_g = 0$.


\begin{lemma}
\label{L-sing1}
If $(k-1)(p-1)\ge1$ and $\beta$ is a $(\dim M  - 1)$-form on $M\setminus\Sigma$ with metric $g$ such that $\int_M\|\beta\|^p\,{\rm d}V_g < \infty$, then $\int_M d\beta = 0$.
\end{lemma}

\begin{proof}
This follows directly from \cite[Lemma~2]{pw1} applied to a vector field $X=\beta^\sharp$ satisfying $\beta = \iota_X{\rm d}V_g$ and the equality $(\Div X)\,{\rm d}V_g=d\beta$.
\end{proof}

For variable pairs $(\omega_t,{\bf T}_t)$ or metrics $g_t$, denote by $\,^{\,\centerdot}{ }\,$ the $t$-derivative at $t=0$ of any $t$-dependent quantity on~$M$.
For $q =2$ we get
$(T_1\wedge T_2)^{\centerdot} = \dot T_1\wedge T_2 + T_1\wedge \dot T_2
= \dot T_1\wedge \widehat{\bf T}_1 -\dot T_2\wedge \widehat{\bf T}_2$.
Thus, in general,
 $\dot{\bf T} = \sum\nolimits_{i}(-1)^{i-1} \dot T_i\wedge \widehat{\bf T}_i$.
If
\begin{equation}\label{E-omega-T-t}
 \iota_{\,{\bf T}_t}\,\omega_t \equiv 1
\end{equation}
then we have
\begin{equation}\label{E-gv-omegat-dot}
 \dot\eta=\iota_{\,{\bf T}}\,d\dot\omega + \iota_{\,\dot{\bf T}}\,d\omega
 = (-1)^{q-1} ({\cal L}_{\,\bf T}\,\dot\omega + {\cal L}_{\,\dot{\bf T}}\,\omega).
\end{equation}

\begin{lemma}
 Let $(\omega_t,{\bf T}_t)\ (|t|\le\eps)$ be a smooth family of $q$-forms and
$q$-vector on $M^{2q+1}\setminus\Sigma$ satisfying \eqref{E-omega-T} and \eqref{E-omega-T-t}, and let ${\cal D}_t=\ker\omega_t$.
Suppose that $g\in{\rm Riem}(M,{\cal D},{\bf T})$~and
\begin{equation}\label{E-L1-sing}
 \int_M\,\|\dot\eta\wedge\eta\wedge (d\eta)^{q-1}\|^p\,{\rm d}V_g<\infty
\end{equation}
for some $p$ such that $(k-1)(p-1)\ge1$. Then
\begin{equation}\label{E-gv-omegat}
 \overset{\centerdot}\gv(\omega,{\bf T}) =  (q+1)\int_M \dot\eta\wedge (d\eta)^q\,.
\end{equation}
Moreover, if the variation satisfies
\begin{equation}\label{E-L1-sing-second}
 \int_M\,\|\ddot\eta\wedge\eta\wedge(d\eta)^{q-1}+2(q-1)\,\dot\eta\wedge d\dot\eta\wedge(d\eta)^{q-2}\|^p\,{\rm d}V_g<\infty,
\end{equation}
then
\begin{equation}\label{E-gv-omegatt}
 \overset{\centerdot\centerdot}\gv(\omega,T) = (q+1)\int_M (\ddot\eta\wedge (d\eta)^q +{q}\,\dot\eta\wedge d\dot\eta\wedge(d\eta)^{q-1} ).
\end{equation}
\end{lemma}

\begin{proof}
We use the Taylor expansions
 $\omega_t = \omega + \dot\omega\,t +\ddot\omega(t^2/2) +O(t^3)$
 and
 $T_i(t) = T_i + \dot T_i\,t + \ddot T_i(t^2/2) +O(t^3)$.
Let $\eta_t =  \iota_{\,{\bf T}_t}\,d\omega_t $.
 Write $\eta_t = \eta + \dot\eta\,t + \ddot\eta(t^2/2) +O(t^3)$.
By the above and the use of
\begin{eqnarray*}
  d(\dot\eta\wedge\eta\wedge(d\eta)^{q-1}) \eq d\dot\eta\wedge\eta\wedge(d\eta)^{q-1} -\dot\eta\wedge d(\eta\wedge(d\eta)^{q-1})\\
 \eq d\dot\eta\wedge\eta\wedge(d\eta)^{q-1} -\dot\eta\wedge (d\eta)^{q},
\end{eqnarray*}
we have
\begin{eqnarray*}
 \eta_t\wedge(d\eta_t)^q \eq\eta\wedge(d\eta)^q + \big((q+1)\dot\eta\wedge(d\eta)^q + d(\eta\wedge(d\eta)^{q-1}\wedge\dot\eta)\big)t \\
 \plus \frac{q+1}2\,\big(\ddot\eta\wedge (d\eta)^q  + q\,\dot\eta\wedge d\dot\eta\wedge(d\eta)^{q-1}\\
 \plus d(\frac{q}2\,\ddot\eta\wedge\eta\wedge (d\eta)^{q-1} +q(q-1)\,\dot\eta\wedge d\dot\eta\wedge(d\eta)^{q-2})\big)\,t^2
  + O(t^3).
\end{eqnarray*}
Let
 $\gv (t) = \int_M \eta_t\wedge (d\eta_t)^q$,
and write $\gv (t) = \gv + t\,\overset{\centerdot}\gv +(t^2/2)\overset{\centerdot\centerdot}\gv + O(t^3)$.
By~the Stokes Theorem, using Lemma~\ref{L-sing1}, \eqref{E-L1-sing} and \eqref{E-L1-sing-second}, we get~\eqref{E-gv-omegat} and \eqref{E-gv-omegatt}.
\end{proof}

From \eqref{E-gv-omegat} and Lemma~\ref{L-02a} we obtain the following.

\begin{proposition}
Let ${\cal D}=\ker\omega$, $g\in{\rm Riem}(M,{\cal D},{\bf T})$, and $(d\eta)^q=0$,
e.g. the normal distribution ${\cal D}^\bot$  harmonic. Then $(\omega,{\bf T})$ is a critical point for
$\gv$
with respect to all variations obeying  \eqref{E-omega-T-t} and \eqref{E-L1-sing}.
\end{proposition}

We will recalculate a general formula \eqref{E-gv-omegatt} (for the second variation of our Godbillon--Vey invariant)
at critical points of our Godbillon--Vey invariant $\gv$.

\begin{lemma}
The following bilinear form on $M\setminus\Sigma$ depending on a $q$-vector ${\bf T}$ and one-form $\eta$:
\begin{equation}\label{E-ITform}
 J(\alpha,\beta)=\int_M
 {\cal L}_{\,\bf T}(
 {\cal L}_{\,\bf T}\,d\,\alpha\wedge(d\eta)^{q-1})\wedge\beta,
\end{equation}
is symmetric on the space of $q$-forms $\alpha,\beta$ on $M$ satisfying
\begin{equation}\label{E-ITform-cond}
 \int_M \|\gamma_1+(-1)^q \gamma_2+(-1)^{q-1} \gamma_3\|^p\,{\rm d}V_g<\infty
\end{equation}
for some $p$ such that $(k-1)(p-1)\ge1$,
where
\begin{eqnarray*}
 \gamma_1 \eq \iota_{\,\bf T}({\cal L}_{\,\bf T}\,d\alpha\wedge(d\eta)^{q-1})\wedge\beta,\quad
 \gamma_2= \iota_{\,\bf T}\,d\alpha\wedge(d\eta)^{q-1}\wedge\iota_{\,\bf T}\,d\beta,\\
 \gamma_3 \eq \alpha\wedge\iota_{\,\bf T}\,({\cal L}_{\,\bf T}\,d\beta\wedge(d\eta)^{q-1}).
\end{eqnarray*}
\end{lemma}

\proof This follows from the following calculation (using Lemmas~\ref{L-iota-ab} and \ref{L-sing1}):
\begin{eqnarray*}
 && {\cal L}_{\,\bf T}({\cal L}_{\,\bf T}\,d\,\alpha\wedge(d\eta)^{q-1})\wedge\beta
 = d\gamma_1 +(-1)^q d\iota_{\,\bf T}\,d\alpha\wedge(d\eta)^{q-1}\wedge\iota_{\,\bf T}\,d\beta\\
 && = d(\gamma_1 +(-1)^q\gamma_2) +(-1)^q d\alpha\wedge\iota_{\,\bf T}({\cal L}_{\,\bf T}\,d\beta\wedge(d\eta)^{q-1}) \\
 && = d(\gamma_1 +(-1)^q\gamma_2+(-1)^{q-1}\gamma_3)
 +{\cal L}_{\,\bf T}({\cal L}_{\,\bf T}\,d\,\beta\wedge(d\eta)^{q-1})\wedge\alpha.\quad\Box
\end{eqnarray*}

\begin{remark}\rm
The form \eqref{E-ITform} serves as the \textit{index form} for our variational problem for integrable distributions ${\cal D}$.
Let a Riemannian metric $g$ be compatible with $(\omega,{\bf T})$, and ${\rm d}V_g$ its volume form.
It defines Hodge star operator on the space of differential forms, $\star_r:\Lambda^r(M)\to\Lambda^{2q+1-r}(M)$, where $0\le r\le 2q+1$.
We~will not decorate $\star$  with $_r$ in what follows. If \eqref{E-ITform-cond} holds, the
self-adjoint Jacobi type operator
$D: \Lambda^q(M) \to \Lambda^q(M)$ corresponding to \eqref{E-ITform} is given by
\[
 D(\alpha)=\star\,{\cal L}_{\,\bf T}({\cal L}_{\,\bf T}\,d\,\alpha\wedge(d\eta)^{q-1}).
\]
\end{remark}

\begin{theorem}\label{T-main1}
Suppose that ${\cal D}=\ker\omega$ be integrable on $M^{2q+1}\setminus\Sigma$, and let $g\in{\rm Riem}(M,{\cal D},{\bf T})$.
Then

\noindent\ \
{\bf (i)} $(\omega,{\bf T})$ is critical for the functional \eqref{E-gv-invar}
with respect to all variations  obeying \eqref{E-omega-T-t}~and
\begin{equation}\label{E-cond1-Sigma}
 \int_M\,\|
 \eta\wedge(d\eta)^{q-1}\wedge\dot\eta -(q+1)\,\dot\omega\wedge\iota_{\,\bf T}\,(d\eta)^q
 \|^p\,{\rm d}V_g<\infty
\end{equation}
for some $p$ such that $(k-1)(p-1)\ge1$,
if and only if the following
holds:
\begin{equation}\label{E-cond0}
 \iota_{\,\bf T}{\cal L}_{\,\bf T}(d\eta)^q =0.
\end{equation}
\ \ {\bf (ii)} a critical pair $(\omega,{\bf T})$ is extremal for \eqref{E-gv-invar} for all variations
obeying \eqref{E-omega-T-t}, \eqref{E-cond1-Sigma}~and
\begin{equation*}
 \int_M\,\| \dot\omega\wedge \iota_{\,\bf T}(
 {\cal L}_{\,\bf T}\,d\,\dot\omega\wedge(d\eta)^{q-1})
 +\ddot\omega\wedge\iota_{\,\bf T}(d\eta)^q \|^p\,{\rm d}V_g<\infty,
\end{equation*}
if and only if the
bilinear~form $J$ in \eqref{E-ITform} is definite for all such variations.
\end{theorem}

\begin{proof}
{\bf (i)}. By integrability conditions, $d\omega=\omega\wedge\eta$,
applying $d$ yields $\omega\wedge d\eta=0$.
Hence, and assuming $\omega=\omega_1\wedge\ldots\wedge\omega_q$, we obtain
 $d\eta = \sum\nolimits_{i=1}^q\omega_i\wedge\alpha_i$
for some one-forms $\alpha_i$.
Thus,
\begin{equation}\label{E-deta-q}
 (d\eta)^q = \omega\wedge\alpha,\quad
  \alpha=\alpha_1\wedge\ldots\wedge\alpha_q.
\end{equation}
Using
\eqref{E-eta} and \eqref{E-gv-omegat-dot},
from \eqref{E-gv-omegat} we get
\begin{eqnarray}\label{E-gv-proof1}
\nonumber
 &&\hskip-17mm (q+1)^{-1} (\eta\wedge(d\eta)^q)^{\,\centerdot}  =  (\iota_{\,\dot{\bf T}}\,d\omega+\iota_{\,\bf T}\,d\dot\omega)\wedge (d\eta)^q+d\beta_1 \\
\nonumber
 && = \iota_{\,\dot{\bf T}}\,(\omega\wedge\eta)\wedge (d\eta)^q +\iota_{\,\bf T}\,d\dot\omega\wedge (d\eta)^q +d\beta_1\\
 && =  -(\iota_{\,\bf T}\,\dot\omega)\,\eta\wedge (d\eta)^q +\iota_{\,\bf T}\,d\dot\omega\wedge (d\eta)^q +d\beta_1,
\end{eqnarray}
where $\beta_1=(q+1)^{-1}\eta\wedge(d\eta)^{q-1}\wedge\dot\eta$. Here we used identity
$\omega^2=0$ for locally decomposable $q$-forms, $\iota_{\,\bf T}\,\dot\omega+\iota_{\,\dot{\bf T}}\,\omega= (\iota_{\,\bf T}\,\omega)^{\,\centerdot}=0$ and
 $\iota_{\,\dot{\bf T}}(\omega\wedge\eta)\wedge (d\eta)^q
 = -(\iota_{\,\bf T}\,\dot\omega)\,\eta\wedge (d\eta)^q$.
Calculating
\begin{eqnarray*}
 && d(\dot\omega\wedge \iota_{\,\bf T}\,(d\eta)^q) = d\dot\omega\wedge \iota_{\,\bf T}\,(d\eta)^q
 +(-1)^q\dot\omega\wedge{\cal L}_{\,\bf T}\,(d\eta)^q,
\end{eqnarray*}
we obtain, using Lemma~\ref{L-iota-ab},
\begin{equation}\label{E-gv-proof2}
 \iota_{\,\bf T}\,d\dot\omega\wedge (d\eta)^q
 = -d\dot\omega\wedge \iota_{\,\bf T}\,(d\eta)^q
 = (-1)^q\dot\omega\wedge{\cal L}_{\,\bf T}\,(d\eta)^q -d\beta_2,
\end{equation}
where $\beta_2=\dot\omega\wedge\iota_{\,\bf T}\,(d\eta)^q$.
Using $\iota_{\,{\bf T}}(\eta\wedge (d\eta)^q)=\eta\wedge \iota_{\,{\bf T}}(d\eta)^q$, see Lemma~\ref{L-iota-ab},
from \eqref{E-gv-proof1} and \eqref{E-gv-proof2} we get
\begin{eqnarray*}
 && (q+1)^{-1} (\eta\wedge(d\eta)^q)^{\,\centerdot}
  = (-1)^q (\iota_{\,\dot{\bf T}}\,\alpha)\,\omega^2\wedge\eta +\dot\omega\wedge{\cal L}_{\,\bf T}\,(d\eta)^q
 {-}(\iota_{\,\bf T}\,\dot\omega)\eta\wedge (d\eta)^q \\
 && +d(\beta_1-\beta_2)
 = \dot\omega \wedge(\eta\wedge \iota_{\,\bf T}(d\eta)^q +(-1)^q{\cal L}_{\,\bf T}(d\eta)^q)
  +d(\beta_1-\beta_2).
\end{eqnarray*}
For a critical pair $(\omega,{\bf T})$ with respect to all variations $\dot\omega$ obeying \eqref{E-omega-T-t},
the above, \eqref{E-cond1-Sigma} and Lemma~\ref{L-sing1} yield $\Omega=0$, where
\begin{equation}\label{E-cond-Omega}
 \Omega:=
 \eta\wedge \iota_{\,\bf T}(d\eta)^q +(-1)^q{\cal L}_{\,\bf T}(d\eta)^q.
\end{equation}
Applying $\iota_{\,\bf T}$ to the equality $\Omega=0$
yields \eqref{E-cond0}.

We claim that $\iota_{\,\bf X}\,\Omega=0$ for any ${\bf X}=X_1\wedge\ldots\wedge X_{q+1}\in\Lambda^{q+1}(TM)$ with~$X_1$ and $X_2$ tangent to ${\cal D}$.
Note that, since $\omega$ is a decomposable $q$-form obeying \eqref{E-omega-T}, $\iota_{\,\bf X}\,d\omega=0$
and $\iota_{\,\bf T}(\omega\wedge\alpha)=\alpha+\mu$,
where the $q$-form $\mu$ satisfies $\omega\wedge\mu=0$,
.
Using this and \eqref{E-deta-q}, we find
\begin{eqnarray*}
  \iota_{\,\bf X}\,\Omega \eq (
  \eta\wedge \iota_{\,\bf T}(\omega\wedge\alpha)
 +(-1)^q{\cal L}_{\,\bf T}
 \,(\omega\wedge\alpha))({\bf X})\\
 \eq
 \iota_{\,\bf X}\,(\eta\wedge(\alpha+\mu)) +(-1)^q\iota_{\,\bf X}\,d(\alpha+\mu).
\end{eqnarray*}
Notice that
\[
 0= d(\omega\wedge\mu)
 = d\omega\wedge\mu +(-1)^q\,\omega\wedge d\mu
 = \omega\wedge(\eta\wedge\mu +(-1)^q\,d\mu).
\]
By \eqref{E-deta-q}, we get
\[
 0= d((d\eta)^q) = d\omega\wedge\alpha +(-1)^q\,\omega\wedge d\alpha
 = \omega\wedge(\eta\wedge\alpha +(-1)^q\,d\alpha).
\]
Therefore,
\begin{eqnarray*}
 0 \eq d(\omega\wedge\mu+(d\eta)^q)({\bf T},{\bf X})
 =\omega\wedge(\eta\wedge(\alpha+\mu) +(-1)^q\,d(\alpha+\mu))({\bf T},{\bf X})\\
 \eq (\eta\wedge(\alpha+\mu) +(-1)^q d(\alpha+\,u))({\bf X})=
 \iota_{\,\bf X}\,\Omega.
\end{eqnarray*}
This proves the claim.
To test vanishing of $(q+1)$-form $\Omega$ on a basis $\{{\bf T},N,{\bf B}\}$ defined in Section~\ref{sec:1-1},
by the above, the remaining case is with at most one vector from ${\cal D}$ among components of ${\bf X}=X_1\wedge\ldots\wedge X_{q+1}$,
hence all $q$ vectors $T_i$, components of ${\bf T}$, participate in ${\bf X}$.
Thus, $\iota_{\,\bf T}\,\Omega=0\,\Leftrightarrow\,\Omega=0$.

\smallskip

{\bf (ii)} We have three independent cases for a pair $(\omega_t,{\bf T}_t)$ obeying \eqref{E-omega-T-t}:

\smallskip
{\bf (ii)}$_1$~$\dot{\bf T}$ is parallel to ${\bf T}$ and $\dot\omega$ is parallel to $\omega$,

\smallskip
{\bf (ii)}$_2$~$\dot T_i\in\mathfrak{X}_{\cal D}\ (1\le i\le q)$ and $\dot\omega=0$, hence $\iota_{\,\dot{\bf T}}\,\omega=0$,

\smallskip
{\bf (ii)}$_3$~$\dot{\bf T}=0$ and $\dot\omega$ is a one-form such that $\iota_{\,\bf T}\,\dot\omega=0$.

\noindent
When ${\cal D}$ is tangent to a foliation, variations {\bf (ii)}$_{1,2}$ do not change the functional \eqref{E-gv-invar},
and only variations
{\bf (ii)}$_3$ are essential.
For such variation $(\omega_t,{\bf T})$,
using $\ddot\eta=\iota_{\,\bf T}\,d\ddot\omega$ and
$\ddot{\bf T}=0$ with
 $\iota_{\,\bf T}\,\ddot\omega=0$,
for a critical pair $(\omega,{\bf T})$, using vanishing of the $q$-form $\Omega$ in \eqref{E-cond-Omega}, we get
\begin{eqnarray*}
 \ddot\eta\wedge (d\eta)^q \eq \iota_{\,\bf T}\,d\ddot\omega\wedge(d\eta)^q
 =(-1)^{q-1}\,d\ddot\omega\wedge\iota_{\,\bf T}(d\eta)^q\\
 \eq \ddot\omega\wedge
 {\cal L}_{\,\bf T}(d\eta)^q +(-1)^{q-1}\,d(\ddot\omega\wedge\iota_{\,\bf T}(d\eta)^q)\\
 \eq (-1)^{q-1}\,\ddot\omega\wedge \eta\wedge \iota_{\,\bf T}(d\eta)^q +(-1)^{q-1}\,d(\ddot\omega\wedge\iota_{\,\bf T}(d\eta)^q).
\end{eqnarray*}
Since
\begin{eqnarray*}
 \ddot\omega\wedge \eta\wedge \iota_{\,\bf T}(d\eta)^q({\bf T},N,{\bf B})
 \eq \|H^\bot\|\cdot\ddot\omega \wedge \iota_{\,\bf T}(d\eta)^q({\bf T},{\bf B})\\
 \eq \|H^\bot\|\cdot\iota_{\,\bf T}\,\ddot\omega\cdot(\iota_{\,\bf T}(d\eta)^q)({\bf B}) =0,
\end{eqnarray*}
the $(2q+1)$-form $\ddot\eta\wedge (d\eta)^q$ is closed.
Thus, compare the integrand in \eqref{E-gv-omegatt},
\begin{eqnarray*}
 && \dot\eta\wedge d\dot\eta\wedge (d\eta)^{q-1}
 = \iota_{\,\bf T}\,d\dot\omega\wedge
 {\cal L}_{\,\bf T}\,d\dot\omega\wedge (d\eta)^{q-1} \\
 && = \dot\omega\wedge
 {\cal L}_{\,\bf T}(
 {\cal L}_{\,\bf T}\,d\,\dot\omega\wedge(d\eta)^{q-1})
 +(-1)^{q-1}d(\dot\omega\wedge\iota_{\,\bf T}(
 {\cal L}_{\,\bf T}\,d\,\dot\omega\wedge(d\eta)^{q-1})).
\end{eqnarray*}
Finally,
\begin{eqnarray*}
 && 2(q+1)^{-1}({\eta\wedge (d\eta)^q})^{\,\centerdot\,\centerdot} = \dot\omega\wedge {\cal L}_{\,\bf T}(
 {\cal L}_{\,\bf T}\,d\,\dot\omega\wedge(d\eta)^{q-1})\\
 \plus (-1)^{q-1} d(
 \dot\omega\wedge \iota_{\,\bf T}(
 {\cal L}_{\,\bf T}\,d\,\dot\omega\wedge(d\eta)^{q-1}) +\ddot\omega\wedge\iota_{\,\bf T}(d\eta)^q).
\end{eqnarray*}
From the above and Lemma~\ref{L-sing1}, the claim follows.
\end{proof}

\begin{remark}\rm
For $q=1$, the bilinear form \eqref{E-ITform} reads as
\[
 J(\alpha,\beta)=\int_M({\cal L}_{\,T})^2 d\,\alpha\wedge\beta,
\]
and \eqref{E-cond0} takes the form $({\cal L}_{\,T})^3\omega=0$, for more details see \cite{rw-gv1}.
\end{remark}

\section{Integrability in average}
\label{sec:av}

In Section \ref{sec:1-2}, we considered variations of integrable distributions among arbitrary ones.
One can restrict the space of distributions under consideration to any ``reasonable" subspace, for example, to integrable ones.
Let us recall that the classical Frobenius Theorem says that the ideal ${\cal J}$ of forms generated by
one-forms $\omega_i$, $i = 1, \ldots, q$, is integrable if and only if $d{\cal J}\subset {\cal J}$, that is when there exist
one-forms $\eta_{ij},\ (i, j = 1,\ldots, q)$, such that $d\omega_i = \sum_j \omega_j\wedge\eta_{ij}$. Then
\[
 \omega_i\wedge \omega_0 = 0,\quad  i = 1\ldots q ,
\]
with $\omega_0 = \bigwedge_{i=1}^q d\omega_i$.
 The above motivates the following.

Let a distribution ${\cal D}=\ker\omega$ be framed, and $\omega=\bigwedge_{\,i=1}^q\omega_i$ with the frame $(\omega_1, \ldots, \omega_q)$ fixed. Set, as before,  $\omega_0 = \bigwedge_{i=1}^qd\omega_i$
and define functionals $J_i$, $i = 1, \ldots, q$
on the space $(\Lambda^1 (M^{2q+1})^q$  of $q$-tuples of $1$-forms, by
\[
 J_i(\omega):=\int_M \omega_i\wedge \omega_0.
\]

\begin{definition}\rm
The space  $\Lambda^q_{\rm av}(M^{2q+1})$ of \textit{$q$-forms integrable in average}
is defined as the following extension of the space of decomposable $q$-forms with integrable
$\ker\omega_i$: $\omega\in\Lambda^q_{\rm av}(M)$ if and only if
\begin{equation}\label{E-cond-int}
 J_i(\omega)=0,\quad i\in\{1,\ldots q\}.
\end{equation}
\end{definition}

Set $\widehat\omega_i=\bigwedge_{j\ne i} \omega_j$
and $\widehat\omega\,_{0i}=\bigwedge_{j\ne i} d\omega_j$  for $i = 1,\ldots , q$,
so that $\omega_0=\widehat\omega\,_{0i}\wedge d\omega_i$ for any $i$.

\begin{theorem}\label{T-crit1}
Let ${\cal D}=\ker\omega$ be an integrable framed distribution on $M^{2q+1}\setminus\Sigma$.
Then $(\omega,{\bf T})$ is critical for
\eqref{E-gv-invar}
for all variations obeying \eqref{E-omega-T-t}, \eqref{E-cond-int} and inequalities
\begin{eqnarray*}
 && \int_M\,\|
 (q+1)^{-1}\eta\wedge(d\eta)^{q-1}\wedge\dot\eta -\dot\omega\wedge\iota_{\,\bf T}\,(d\eta)^q
 \|^p\,{\rm d}V_g<\infty,\\
 && \int_M\,\|\sum\nolimits_j \omega_i\wedge \overset{\centerdot}\omega_j\wedge \widehat\omega\,_{0j}\|^p\,{\rm d}V_g<\infty
\end{eqnarray*}
for some $g\in{\rm Riem}(M,{\cal D},{\bf T})$ and $p$ such that $(k-1)(p-1)\ge1$,
if and only if
\begin{equation}\label{E-new-eqs}
\widehat\omega\,_i\wedge\Omega =\sum\nolimits_j\lambda_j
\,(\delta_{ij}\,\omega_0 -d\omega_j\wedge\widehat\omega\,_{0i}),\quad i\in\{1,\ldots,q\}.
\end{equation}
for some constants $\lambda_j\in\RR$ and $i = 1, \ldots , q$.
\end{theorem}

\proof
The proof of Theorem~\ref{T-main1} together with the formula
$(\alpha, \beta)_{L^2} = \int_M \alpha\wedge\star\,\beta$
defining the inner product in the space of forms, shows that the form $\Omega$ of (\ref{E-cond-Omega}), or rather its Hodge star image $\star\,\Omega$, can be considered
as the gradient of the functional $\gv$ at $\omega$. Its component (in the L$^2$-space of $q$-forms) tangent to the domain $(\Lambda^1 (M^{2q+1})^q$ of functionals $J_i$ coincides with the sequence $(\star\,(\widehat\omega\,_1\wedge\Omega), \ldots,
\star\,(\widehat\omega\,_q\wedge\Omega))$ of $q$ one-forms.
 For any $i$ we have
\[
 (\omega_i\wedge \omega_0)^{\,\centerdot} = \overset{\centerdot}\omega_i\wedge \omega_0
 +\sum\nolimits_j\overset{\centerdot}\omega_j\wedge d\omega_i\wedge \widehat\omega\,_{0j}
 -d\big(\sum\nolimits_j \omega_i\wedge \overset{\centerdot}\omega_j\wedge \widehat\omega\,_{0j}\big).
\]
This shows that the system $\{\star\,(\delta_{i1}\,\omega_0
 -d\omega_j\wedge\widehat\omega\,_{0i}), \ldots, \star\,(\delta_{qj}\,\omega_0
 -d\omega_j\wedge\widehat\omega\,_{0i})\,\}$ of $q$ one-forms can be considered as
 the gradient (in the L$^2$-space $\Lambda^1(M^{2q+1})^q$) of $J_i$.
By the Lagrange multipliers method, a point $(\omega,{\bf T})$ is critical for
the functional $\gv$ restricted to the space $\Lambda^1 (M^{2q+1})^q$ if and only if
the gradient of $\gv |_{\Lambda^1 (M^{2q+1})^q}$ coincides with a linear combination (with constant coefficients) of the gradients of $J_i$.
 This is equivalent to the statement of our Theorem.
\hfill$\square$

\begin{remark}\rm
Condition $J_i(\omega)=0$ is weaker than the pointwise condition
$\omega_i\wedge \omega_0=0$, which also is weaker than integrability of
$\ker\omega_i$.
Certainly, (\ref{E-new-eqs}) yields
\begin{equation}\label{E-LT_ge}
 \star\,(\widehat\omega\,_i\wedge\Omega)\wedge\bigwedge\nolimits_{j=1}^q \star\,(\delta_{ij}\,\omega_0-d\omega_j\wedge\widehat\omega\,_{0i}) = 0,\quad i\in\{1,\ldots,q\}.
\end{equation}
but the converse is not true: (\ref{E-LT_ge}) implies  that the forms $\star\,(\widehat\omega\,_i\wedge\Omega)$ are linear combinations of
$\star\,(\delta_{ij}\,\omega_0 - d\omega_j\wedge\widehat\omega\,_{0i})$ but the coefficients may depend on $i$ and vary over $M$.
For $q=1$, \eqref{E-new-eqs} reads as
\begin{equation}\label{E-LT_ge-1}
 ({\mathcal L}_{\,T})^3\,\omega = \lambda\,{\mathcal L}_{\,T}\,\omega,
\end{equation}
and \eqref{E-LT_ge} tells us that
the three-form on $M^3$ representing $\gv(\omega,T)$ is invariant under the flow of $T$,
that is, see also \cite{rw-gv1},
\[
 {\cal L}_{\,T}(\eta\wedge d\eta)=0,
\]
and \eqref{E-cond-int} reads as condition for one-form $\omega$:
\begin{equation}\label{E-cond-int-1}
 \int_M \omega\wedge d\omega =0.
\end{equation}
\end{remark}

It is rather difficult to find explicitly the derivative of the functional \eqref{E-gv-invar} for all variations among foliations,
see \cite{as2015}.
Therefore, in the next corollary we present just a condition sufficient for being critical point (foliation) of \eqref{E-gv-invar} with respect to such variations.

\begin{corollary}
Let $\calf$ with $T\calf=\ker\omega$ be a codimension $q$ foliation of $M^{2q+1}$.
If \eqref{E-new-eqs}
holds for any $(\omega,{\bf T})$ such that $\iota_{\,\bf T}\,\omega=1$,
then $\gv(\calf)$ is infinitesimally rigid, i.e., $\overset{\centerdot}\gv(\calf)=0$ for any infinitesimal deformation of $\calf$ among foliations.
\end{corollary}

\section{Examples}


A distribution ${\cal D}$ is called \textit{mixed totally geodesic} with respect to sub-distributi\-ons
$\{N\}$ and ${\cal B}$ if $h(N,{\cal B})=h({\cal B},N)=0$.
The simple examples are provided by  totally umbilical
${\cal D}$'s, see Definition~\ref{D-t-umb}.


\begin{proposition}\label{P-Riem}
Let $g$ be a Riemannian metric
and ${\cal D}=\ker\omega$ a mixed totally geodesic distribution
with respect to $\{N\}$ and ${\cal B}$ on $M^{2q+1}\setminus\Sigma$.
If either ${\cal D}^\bot$ or ${\cal B}$ is integrable and $H^\bot$ $($or $N)$ is parallel along ${\cal D}^\bot$
then $(\omega,{\bf T})$ is critical for \eqref{E-gv-invar}
with respect to all variations obeying \eqref{E-omega-T-t} and \eqref{E-cond1-Sigma}.
\end{proposition}

\begin{proof}
By integrability of either ${\cal D}^\bot$ or ${\cal B}$ and Theorem~\ref{T-00},
equality \eqref{eq:pw2013} holds.
Thus, conditions
yield vanishing of $\iota_{\,\bf T}(d\eta)^q$.
By
Theorem~\ref{T-main1},
$(\omega,{\bf T})$ is critical for \eqref{E-gv-invar}, when $\iota_{\,\bf T}(d\eta)^q=0$.
\end{proof}

\begin{corollary}
Let ${\cal D}=\ker\omega$  on $M^{2q+1}$ with a metric $g$,
${\cal D}^\bot$ tangent to a Riemannian foliation and $N$ parallel along its leaves.
Then $(\omega,{\bf T})$ is critical for \eqref{E-gv-invar}
with respect to all variations obeying \eqref{E-omega-T-t}.
\end{corollary}

Let $(B,g_B)$ and $(F,g_F)$ be Riemannian manifolds and $\phi>0$ a smooth function on $B\times F$.
The~\textit{twisted product} $M=B\times_\phi F$ is the manifold $M=B\times F$ with the metric
 $g = \pi^*g_B + (\phi\circ \pi)^2(\pi')^*g_F$, where $\pi:M\to B$ and $\pi':M\to F$ are projections.
The fibers $\{x\}\times F\ (x\in B)$ are totally umbilical, and the leaves $B\times\{y\}\ (y\in F)$ are totally geodesic.
If we regard $\pi:B\times_\phi F\to B$ as a submersion, then the fibers are conformally related with each other;
this gives us a \textit{conformal submersion}. If $\phi$ depends on $B$ only, the twisted product becomes the \textit{warped~product}.

\begin{lemma}[see \cite{c-by}]\label{L-chen1}
Let $M=B\times_\phi F$ be a twisted product. Then

(i) fibers $\{x\}\times F$ are totally umbilical in $M$ with the mean curvature vector $-(\nabla\log\phi)^\top$,

(ii) fibers have parallel mean curvature if and only if $\phi=\phi_1\phi_2$ with $\phi_1\in C^2(B)$ and $\phi_2\in C^2(F)$.
\end{lemma}

\begin{corollary}
Let $B^{q+1}\!\times_\phi F^q$ be a twisted product,
${\bf T}=T_1\wedge\ldots\wedge T_q$, where $\{T_i\}$ are tangent to the fibers, and ${\cal D}$ is tangent to the leaves.
(i)~If $\phi$ is the product of functions $\phi_1\in C^2(B)$ and~$\phi_2\in C^2(F)$
then $(\omega,{\bf T})$ is critical for \eqref{E-gv-invar} with respect to all variations obeying \eqref{E-omega-T-t}.
(ii)~In particular, if $B^{q+1}\times_\phi F^q$ is a warped product, then $(\omega,{\bf T})$
is critical for \eqref{E-gv-invar} and $\gv(\omega,{\bf T})=0$.
\end{corollary}

\begin{proof} By our conditions, the leaves $\bar M\times\{y\}$ are totally geodesic: $h=0$.
Let the fibers $F^q\times\{y\}$ have mean curvature vector $H^\bot$.
Thus the claims follow from Proposition~\ref{P-Riem} and Lemma~\ref{L-chen1}.
\end{proof}

Let a function $f(r)\ (r\ge0)$ of class C$^k$ has vertical asymptote at $r=r_0>0$ and satisfies $f(0) = 0$.
The foliation within a Reeb component in the solid torus $D^2\times S^1$ can be defined by the equation $\omega=0$,~where
\[
 \omega(r,t)=\cos\mu(r)\,dt - \sin\mu(r)\,dr,\quad
 \mu(r)=\arctan f'(r),
\]
$(r,\theta)$ are the polar coordinates in the disc $D^2=\{0\le r\le r_0\}$ and $t$ is a parameter along the circle~$S^1$,
see e.g.~\cite{tamura,walczak}. Since $\mu(0)=0$ and $\mu(r_0)=\pi/2$,
then $ \omega(0,t)=dt$ (hence $ {\cal D}(0,t)=\{\partial_r, \partial_\theta \} $ -- tangent plane to $D^2$ at the origin)
and  $ \omega(r_0,t)=-dr$ (hence $ {\cal D}(r_0,t)=\{\partial_t, \partial_\theta \} $).
Gluing two foliated solid tori yields a Reeb foliation of a sphere~$S^3$.
We will show that critical foliations have singularity set $\Sigma$, the axis $r=0$, in this case we assume $f'(0)\ne0$.

It seems reasonable that critical foliations are singular.
It is a bit like minimizing the total curvature of a curve with the axis of symmetry leads to the graph of
$x\mapsto |x|$, which is not differentiable at 0.

The following result is a suitably modified  analogue of the one in~\cite{rw-gv1}.


\begin{proposition}
The Reeb foliation of $S^3$ produced by a function $f=f(r)\ (r\ge0)$ is critical for the action
 $\gv:(\omega,T) \mapsto \int_M \eta\wedge d\eta$

\noindent\
$(i)$~in general if and only if $f$ solves the following Cauchy's problem with real parameters $A_0,A_1,A_2:$
\begin{eqnarray}\label{E-cond2}
\nonumber
 && f''' =
 \frac{2 ((f')^2-1)}{(1+(f')^2)f'}\,(f'')^2 + \frac{A_0(1+(f')^2)^{5/2}}{(f')^3}
  , \\
 && f(0)=0,\quad f'(0)=A_1\ne0,\quad f''(0)=A_2.
\end{eqnarray}

\noindent\
$(ii)$~and all variations obeying \eqref{E-cond-int-1}, if and only if $f$ solves the following Cauchy's problem
with real parameters $\lambda,A_1,A_2$ and $A_3:$
\begin{eqnarray}\label{E-cond3}
\nonumber
 && f^{(4)} = \frac{(6(f')^2-7)}{f'(1+(f')^2)}\,f''' f''
 -\frac{2(3(f')^4 {-}9(f')^2+2)}{(1+(f')^2)^2(f')^2}\,(f'')^3
 +\lambda\,\frac{(1+(f')^2)}{(f')^2}\,f'' ,\\
 && f(0)=0,\quad f'(0)=A_1\ne0,\quad f''(0)=A_2,\quad f'''(0)=A_3.
\end{eqnarray}
\end{proposition}

\proof
Set~$T(r,t)=\cos\mu(r)\,\partial_t-\sin\mu(r)\,\partial_r$, then $\omega(T)\equiv1$ in $M=D^2\times S^1$.
First we compute
\begin{equation*}
 d\omega = -\mu'\sin\mu(dr\wedge dt),\quad
 \iota_{\,T} (dr\wedge dt) = -\cos\mu\,dr-\sin\mu\,dt.
\end{equation*}
Then we observe that $\gv(\omega,T)=0$:
\[
 \eta =\iota_{\,T}\, d\omega = \mu'\sin\mu(\cos\mu\,dr+\sin\mu\,dt),\quad
 d\eta =(\mu'\sin^2\mu)'\,dr\wedge dt,
\]
therefore,
 $\eta\wedge d\eta =0$.
To verify \eqref{E-cond0} with $q=1$, we then find
\begin{eqnarray*}
 && (\iota_{\,T}\, d)^2\omega = -(\mu'\sin^2\mu)'(\cos\mu\,dr+\sin\mu\,dt),\\
 && (\iota_{\,T}\, d)^3\omega = \big((\mu'\sin^2\mu)'\sin\mu\big)'(\cos\mu\,dr+\sin\mu\,dt).
\end{eqnarray*}

(i)~According to \eqref{E-cond0} with $q=1$, a pair $(\omega,T)$ is critical for the action
$\gv:(\omega,T) \mapsto \int_M \eta\wedge d\eta$
if and only if
$((\mu'\sin^2\mu)'\sin\mu)'\equiv0$ for $r\ge0$, that is
 \begin{equation}\label{E-cond-A1}
 (\mu'\sin^2\mu)'\sin\mu = A_0
\end{equation}
for some $A_0\in\RR$.
The ODE \eqref{E-cond-A1}, using $\mu=\arctan f'$ and $\mu'=f''/(1+(f')^2)$, can be rewritten in terms of Cauchy's problem \eqref{E-cond2}, which has a unique solution.
This way we get a family (depending on $f'(0)=A_1\ne0$) of solutions of \eqref{E-cond2},
see graphs (obtained by Maple program) on Figure~\ref{F-2-7-12} with the value $r_0$ depending on $A_0$.
If $A_0=0$ then
\eqref{E-cond-A1} reduces~to
\begin{equation}\label{E-cond-A10}
 \mu'\sin^2\mu = \tilde A_0
\end{equation}
for another constant $\tilde A_0\in \RR$.
This ODE has the following integral: $ 2\mu -\sin(2\mu) =4\tilde A_0 r +C$.

Notice that $\mu\ne\const$, hence $\tilde A_0\ne0$, because if $\mu = k =\const$ then
$f(r)=(\tan k) r$ has no asymptotes for $r>0$ and does not produce critical foliation.
For $f$, \eqref{E-cond-A10} provides the ODE,
 $f'' = \tilde A_0\big(\frac{(1+(f')^2)}{f'}\,\big)^2$,
with similar to Figure~\ref{F-2-7-12} graphs of solutions with the value $r_0$ depending on $\tilde A_0$.

(ii)~By the above, $(\iota_{\,T}\, d)^3\omega$ is parallel to $\eta$, and \eqref{E-LT_ge-1} holds if the ratio is constant,
\[
 \big((\mu'\sin^2\mu)'\sin\mu\big)' = \lambda \mu'\sin\mu\quad \mbox{\rm for some}\ \lambda\in\RR.
\]
From this, with the little aid of Maple calculations, we yield \eqref{E-cond3}.
\hfill$\square$

\begin{figure}[t]
\begin{center}
\includegraphics[scale=0.68,angle=0]{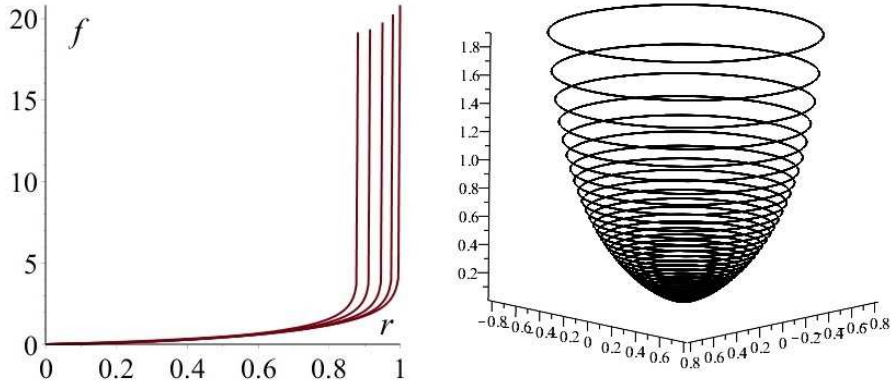}
\caption{\label{F-2-7-12}\small Family of solutions $f(r)$ to \eqref{E-cond2} with
$A_0=1,A_2=0$ and $A_1=i/8\ (i=1\ldots 5)$, producing singular Reeb foliations by rotation about $f$-axis.
 }
\end{center}
\end{figure}


\begin{remark}[The Bott invariant of transversely holomorphic flows]\rm
Let $Y$ be a nonzero on $M^{2q+1}\setminus\Sigma$ vector field.
The flow generated by $Y$ is \textit{transversely holomorphic} if there is a complex structure $J$
on the $2q$-plane bundle $TM/\<Y\>$ invariant under the flow of $Y$.
 Assume that $TM/\<Y\>$ is trivial and there is a pair of pointwise linearly independent decomposable $q$-forms $\omega_1,\omega_2$ on $M\setminus\Sigma$
with a common kernel $\ker\omega_1\cap\,\ker\omega_2$ spanned by $Y$, and such real $2q$-form $\omega_1\wedge\omega_2$ defines the transverse orientation.
If the flow generated by $Y$ is {transversely holomorphic} then
the complex-valued $q$-form $\omega_c=\omega_1+i\,\omega_2$ is \textit{formally integrable} \cite{gp}, i.e.,
\begin{equation}\label{E-C1}
 \omega_c\wedge d\omega_c = 0
 \quad\Longleftrightarrow\quad
 \bigg\{\begin{array}{c}
   \omega_1\wedge d\omega_1 = \omega_2\wedge d\omega_2\,, \\
   \omega_1\wedge d\omega_2 = -\omega_2\wedge d\omega_1\,.
 \end{array}
\end{equation}
For a complex-valued vector field $T_c=T_1+i\,T_2$ we may assume $\omega_j(i\,T_k)=i\,\omega_j(T_k)$, then
\begin{equation}\label{E-omega-Tc}
 \iota_{\,T_c}\,\omega_c=1
  \quad\Longleftrightarrow\quad
 \bigg\{\begin{array}{c}
   \iota_{\,T_1}\,\omega_1 +\iota_{\,T_2}\,\omega_2 = 1\,, \\
   \iota_{\,T_1}\,\omega_2 +\iota_{\,T_2}\,\omega_1 = 0\,.
 \end{array}
\end{equation}
If \eqref{E-C1} holds, then there is a complex-valued one-form $\eta_c=\eta_1+i\,\eta_2$ such that
\begin{equation}\label{E-domega-c}
 d\omega_c=\omega_c\wedge\eta_c,
\end{equation}
moreover, the complex-valued one-form $\eta_c$ in \eqref{E-domega-c} can be chosen by
\begin{equation}\label{E-eta-c}
 \eta_c = \iota_{\,T_c}\,d\omega_c.
\end{equation}
Indeed, we have
\[
 0=\iota_{\,T_c}\,(d\omega_c\wedge\omega_c)=\iota_{\,T_c}\,d\omega_c\wedge\omega_c+d\omega_c.
\]
This and \eqref{E-eta-c} yield \eqref{E-domega-c}.
For formally integrable $\omega_c$, the complex number $\int_M\eta_c\wedge (d\eta_c)^q$, called the \textit{Bott invariant} of the flow of $Y$, is independent of choices.
For a pair $(\omega_c,T_c)$ obeying \eqref{E-omega-Tc}, define
\begin{equation*}
 \gv(\omega_c,T_c):=\int_M\eta_c\wedge (d\eta_c)^q.
\end{equation*}
The following fundamental example belongs to R.~Bott, see e.g. \cite{as2010}.
Consider a holomorphic vector field $X_\lambda=\sum_j \lambda_j z_j \,\partial_{z_j}$ on $\CC^{q+1}$,
where $(z_0,\ldots,z_q)$ are standard coordinates and $\lambda_i\ne0$ are complex numbers.
Let the convex hull of $\lambda_0,\ldots,\lambda_q$ does not contain the origin.
A foliation of $\CC^{q+1}$ by orbits of $X$ induces a one-dimensional foliation (flow) $\calf_\lambda$ of the unit sphere $S^{2q+1}$.
This flow is transversely holomorphic of complex dimension $q$, and
 $\gv(\calf_\lambda)=(\sum\nolimits_j \lambda_j)^{q+1}/\prod\nolimits_j \lambda_j$.
Hence, the Bott invariant (of the flow above) is non-trivial and admits continuous variations.
However, the Godbillon--Vey invariant is rigid under both actual and infinitesimal deformations in the category
of transversely holomorphic foliations, see \cite{as2010}.
 Notice that for real $q$-forms $\alpha,\beta$, the following equalities hold:
 $d(\alpha+i\,\beta)=d\alpha+i\,d\beta$,
 \
 $({\alpha+i\,\beta})^{\,\centerdot}=\overset{\centerdot}\alpha+i\,\overset{\centerdot}\beta$,
 \
 $(\alpha+i\,\beta)^\sharp =\alpha^\sharp+i\,\beta^\sharp$,
 and
 $\iota_{X_1+i\,X_2}\,\alpha = \iota_{X_1}\alpha + i\,\iota_{X_2}\alpha$,
and Lemma~\ref{L-sing1} is valid for $\beta_c=\beta_1+i\,\beta_2$ with $\|\beta\|^2=\|\beta_1\|^2+\|\beta_2\|^2$.
Thus, the results of Sections~\ref{sec:1-1}, \ref{sec:1-2} and \ref{sec:av} (Theorems~\ref{T-crit1} and \ref{T-00}) are valid for complex-valued forms and vector~fields.
\end{remark}

\section{Metric formula for the Godbillon--Vey type invariant}

The~following Godbillon--Vey type functional is defined
for any Riemannian metric on~$(M,{\cal D})$:
\begin{equation}\label{E-gv-invar-g}
 \gv_{\cal D}: g \mapsto -\int_M \|H^\bot\|\cdot(d\eta)^q({\bf T},{\bf B})\,{\rm d}V_g.
\end{equation}
Here
${\bf B}=B_1\wedge\ldots\wedge B_q$,
where $\{B_j\}$ is a local orthonormal frame of~${\cal B}$ such that $({\bf T},N,{\bf B})$ is positive oriented.
This $\gv_{\cal D}$ helps us to study the \textbf{question}:
\textit{What are the best in a sense metrics on a manifold endowed with a distribution ${\cal D}$}?
Such metrics are proposed to be among critical metrics of the action~\eqref{E-gv-invar-g}.
 The~$2q$-form $(d\eta)^q$ can be expressed on $U$ in terms of extrinsic geometry of ${\cal D}$ and~${\cal D}^\bot$.

\begin{theorem}
\label{T-00}
If either ${\cal D}^\bot$ or ${\cal B}$ is integrable then the Godbillon--Vey class of $(\omega,{\bf T})$
can be represented by the form $\eta\wedge (d\eta)^q$ given by
\begin{equation}\label{eq:pw2013}
 (\eta\wedge (d\eta)^q)({\bf T},N,{\bf B})
 =-\|H^\bot\|^{q+1}\det( \<\nabla_{T_i}N,B_j\> - \<h_{B_j,N},T_i\> ).
\end{equation}
\end{theorem}

\begin{proof}
If either ${\cal D}^\bot$ or ${\cal B}$ is integrable,  then -- respectively --
either $d\eta(T_i,T_j)$ or $d\eta(B_i,B_j)$ vanish, see Lemma~\ref{L-02} in what follows. Hence,
\begin{equation*}
 (d\eta)^q({\bf T},{\bf B}) = \sum\nolimits_{\sigma\in S_q} {\rm sign}(\sigma)\cdot
 d\eta(T_1,B_{\sigma_1})\cdot\ldots\cdot d\eta(T_q, B_{\sigma_q}),
\end{equation*}
where $S_q$ is the group of all permutations of the set $\{1,\ldots,q\}$.
By Lemma~\ref{L-02} again,
\eqref{eq:pw2013} holds.
\end{proof}

\begin{remark}\rm
For $q=1$, the vector field $T$ spans ${\cal D}^\bot$ on $(M^3,g)$.
Let $\{T,N,B\}$ be the Fren\'{e}t frame, and $k$ and $\tau$ the curvature and the torsion of $T$-curves.
Then \eqref{eq:pw2013} reads as
\[
 (\eta\wedge d\eta)(T,N,B)=-k^2(\tau - h_{B,N}),
\]
see \cite{rw-gv1} and, for integrable ${\cal D}$, \cite{rw73}.
\end{remark}

\begin{remark}\rm
In \cite{rw-1} we introduced geometric invariants for $n$-tuples, $n\in\mathbb{N}$, of square matrices, or, endomorphisms.
Given $n$ such matrices ${\bf A}=(A_1,\ldots,A_n)$, consider the polynomial of $n$ variables,
\begin{equation}\label{E-sigmaAn}
 P({\bf A})(t_1,\ldots,t_n)=\det(\id+t_1A_1+\ldots+t_nA_n);
\end{equation}
the coefficients
$\sigma_\lambda({\bf A}), \lambda=(\lambda_1,\ldots,\lambda_n)\ge0$
at $t_1^{\lambda_1}\cdot\ldots\cdot t_n^{\lambda_n}$ are invariants of our set of~matrices.
Set $|\lambda|=\lambda_1+\ldots+\lambda_n$. Taking $t_1=\ldots=t_n=t$ and comparing the coefficients at $t^q$, this yields
\begin{equation*}
 \det(A_1+\ldots+A_n) = \sum\nolimits_{|\lambda|=q}\sigma_\lambda({\bf A}).
\end{equation*}
Formula \eqref{eq:pw2013} allows to express the form $\eta\wedge(d\eta)^q$
in terms of invariants $\sigma_\lambda(A_1,A_2)$ of \textit{two liner transformations} depending on the extrinsic geometry
of the almost product structure $({\cal D},{\cal D}^\bot)$ on~$(M,g)$:
\[
 A_1:X \to  \nabla^\top_X N,\quad
 A_2:Y \to \nabla^\top_Y N = -h(Y,N);
\]
hereafter, $\nabla^\top$ denotes the connection in the bundle ${\cal D}$ generated by the Levi-Civita connection on
$(M,g)$. These maps transform one of the spaces ${\cal B}$ and ${\cal D}^\bot$, into another and are
represented in positive oriented orthonormal frames by $q$-by-$q$ matrices. With this notation,
using \eqref{E-sigmaAn} with $n=2$, i.e., $\det(A_1+A_2)=\sum\nolimits_{k+l=q}\sigma_{k,l}(A_1,A_2)$, \eqref{eq:pw2013} reads as
\[
 (\eta\wedge(d\eta)^q)({\bf T},N,{\bf B}) =-\|H^\bot\|^{q+1}\sum\nolimits_{k+l=q}\sigma_{k,l}(A_1,A_2).
\]
\end{remark}

The~\textit{integrability tensor} ${\cal T}$ of ${\cal D}$ (vanishing for ${\cal D}$ tangent to a foliation)~is
\[
 2\,{\cal T}_{X,Y} = [X,Y]^\bot = h_{X,Y}-h_{Y,X},\quad X,Y\in{\cal D}.
\]
Let $B_i\ (i\le q)$ be an orthonormal local basis of ${\cal B}$.
Set $\Div^\bot Q =\sum\nolimits_{j} \<\nabla_{T_j}Q, T_j\>$ for a tensor field~$Q$.

\begin{lemma}\label{L-02}
The 2-form $d\eta$ attains the following values on $U$:
\begin{eqnarray}\label{E-d-eta}
\nonumber
 && (-1)^{q-1} d\eta(N,B_i) = \|H^\bot\|\,\<\nabla_N\,N, B_i\> -B_i(\|H^\bot\|) \\
\nonumber
 && = 2 \<h_{N,N},{\cal T}_{N,B_i}\> -2\Div^\bot {\cal T}_{N,B_i}  \\
 && + 2\sum\nolimits_{j}\big( \<{\cal T}_{(\nabla_{T_j}N)^\top,B_i}, T_j\> +\<h_{B_j,B_i},{\cal T}_{N,B_j}\> - \<h_{B_j,N},{\cal T}_{B_i,B_j}\>\big)
\end{eqnarray}
and
\begin{eqnarray}\label{E-d-eta-2}
\nonumber
 && (-1)^{q-1} d\eta(T_i,B_j) = \|H^\bot\|( \<\nabla_{T_i}N, B_j\>- \<h_{B_j,N},T_i\>),\\
\nonumber
 &&   (-1)^{q-1} d\eta(B_i,B_j) = -\|H^\bot\|\cdot \<[B_i,B_j],N\>, \\
 \nonumber
 && (-1)^{q-1} d\eta(T_i,T_j) = -\|H^\bot\|\cdot \<[T_i,T_j],N\>,\\
 && (-1)^{q-1} d\eta(T_i,N) = T_i(\|H^\bot\|) - \|H^\bot\| \<h_{N,N}, T_i\>.
\end{eqnarray}
\end{lemma}

\begin{proof}  Using \eqref{eq1:pw2013}, one gets the first equality of \eqref{E-d-eta},
\begin{eqnarray*}
 d\eta (N, B_i) \eq N(\eta(B_i)) - B_i(\eta(N)) - \eta([N,B_i])\\
 \eq  (-1)^{q-1}\big(\|H^\bot\|\,\<\nabla_N N, B_i\> - B_i(\|H^\bot\|)\big).
\end{eqnarray*}
Differentiating $\<[N, B_i], T_j\>$ in the $T_j$-direction, after a lengthy calculation using symmetries of the curvature tensor $R$, yields
\begin{eqnarray*}
 && \sum\nolimits_{j} T_j\,(\<[N,B_i],T_j)\> = \<[N,B_i], H^\bot\>
 +\sum\nolimits_{j} \<\nabla_N\nabla_{T_j}\,B_i\\
 &&  +\nabla_{[T_j,N]}B_i +R_{T_j,N}B_i -\nabla_{B_i}\nabla_{T_j}\,N -\nabla_{[T_j,B_i]}\,N -R_{T_j,B_i}N,\ T_j\> \\
  && =\|H^\bot\|\,\<[N,B_i], N\> + \sum\nolimits_{j} \< -\nabla_N(\tau_{ji}N)
  +\nabla_{ A_jN  + \sum\nolimits_{l}(\tau_{jl}B_l-k_{jl} T_l)}B_i \\
 &&  +\nabla_{B_i}(\sum\nolimits_{l}k_{jl}T_l) -\nabla_{B_i}(\sum\nolimits_{l}\tau_{jl}B_l) -\nabla_{A_jB_i -\tau_{ji}N}\,N -R_{N,B_i}T_j,\ T_j\> \\
 && = B_i(\|H^\bot\|) - \|H^\bot\|\,\<\nabla_N  N, B_i\> + 2 \<h_{N,N},{\cal T}_{N,B_i}\> \\
 && +\, 2\sum\nolimits_{j}\big( \<{\cal T}_{(\nabla_{T_j}N)^\top,B_i}, T_j\>
 +\<h_{B_j,B_i},{\cal T}_{N,B_j}\> -\<h_{B_j,N},{\cal T}_{B_i,B_j}\>\big).
\end{eqnarray*}
Here we used
$A_j X:=-(\nabla_X T_j)^\top\ (X\in{\cal D})$ and Fren\'{e}t type formulas for ${T}_i$-derivatives:
\begin{eqnarray*}
 \nabla_{T_i}\,T_j \eq -k_{ij}N+\sum\nolimits_{k}s_{ijk}B_k,\quad
 \nabla_{T_i}\,N = -k_{ij}T_j+\sum\nolimits_{j}\tau_{ij}B_j,\\
 \nabla_{T_i}\,B_j \eq -\tau_{ij}N -\sum\nolimits_{k}s_{ikj}B_k
\end{eqnarray*}
with certain functions $k_{ij}$ and $s_{ikj}$. Note that $\sum_i k_{ii}=\< H^\bot, N\>$.
From this, the definition of ${\cal T}$ and
\[
 2\Div^\bot{\cal T}_{N,B_i} = 2\sum\nolimits_{j} \<\nabla_{T_j}({\cal T}_{N,B_i}), T_j\> = \sum\nolimits_{j} T_j(\<[N,B_i]^\bot, T_j)\>
\]
we deduce the second equality of \eqref{E-d-eta}. Next,
\[
 d\eta(T_i,B_j)=T_i(\eta(B_j))-B_j(\eta(T_i))-\eta([T_i,B_j])=(-1)^{q}\|H^\bot\|\,\<[T_i,B_j],N\>,
\]
from which \eqref{E-d-eta-2}$_{1}$ follows. The~proofs of \eqref{E-d-eta-2}$_{2,3,4}$ are also straightforward.
\end{proof}

\begin{lemma}\label{C-02}
Let ${\cal D}=\ker\omega$ be a codimension $q$ distribution on $M^{2q+1}$ and $g\in{\rm Riem}(M,{\cal D},{\bf T})$.

$(i)$~If $T_i(t)=C^j_i(t)\,T_j$ and $\omega(t)=\det C(t)^{-1}\omega$ for some $C(t):M\to GL(q,\RR)$ such that $C(0):M\to{\rm id}_q$,  then
\begin{eqnarray}\label{eq:rel_gv1}
\nonumber
 (\eta\wedge (d\eta)^q)^{\,\centerdot} \eq (q+1)\sum\nolimits_{i}(-1)^{i} T_i(\tr\dot C)
 \,(d\eta)^q(\widehat{\bf T}_i,N,{\bf B})\,{\rm d}V_g  \\
 && +\,d(\dot\eta\wedge\eta\wedge(d\eta)^{q-1} +(q+1)\,(\tr\dot C)\,(d\eta)^q) .
\end{eqnarray}

$(ii)$~If $T_i(t)=T_i+X_i(t)$ for $X_i(t)\in\mathfrak{X}_{\cal D}\ (|t|<\eps)$ such that $X_i(0)=0$
and $\omega(t)=\omega$, then
\begin{eqnarray}\label{eq:d_gv2}
\nonumber
 && ({\eta\wedge (d\eta)^q})^{\,\centerdot} = (q+1)\big[ d\omega(\dot{\bf X},N)\,(d\eta)^{q}({\bf B},{\bf T})\\
 && +\sum\nolimits_{i}(-1)^{i-1} d\omega(\dot{\bf X},T_i)\,(d\eta)^{q}(\widehat{\bf T}_i,N,{\bf B}) \\
\nonumber
 && + \sum\nolimits_{k}(-1)^{q+k}d\omega(\dot{\bf X},B_k)\,(d\eta)^{q}({\bf T},N,\widehat{\bf B}_k)
 \big]\,{\rm d}V_g -d(\eta\wedge(d\eta)^{q-1}\wedge\iota_{\,\dot{\bf X}}\,d\omega).
\end{eqnarray}
\end{lemma}

\begin{proof} Notice that in both cases, (i) and (ii), \eqref{E-omega-T-t} holds.

(i)~We have $\dot T_i=\dot C^j_i\,T_j$ and $\dot\omega=-\dot c\,\omega$, where $ c(t)=\det C(t)$.
By Jacobi's formula, that expresses the derivative of
$\det C$ in terms of the adjunct of $C$ and the derivative of~$C$, and conditions, we have $\dot c=\tr\dot C$.
Since \eqref{E-gv-omegat-dot} and $\iota_{\,N}\,\omega=\iota_{\,B_k}\,\omega=0$, the following equalities provide~\eqref{eq:rel_gv1}:
\begin{eqnarray*}
 ({\eta\wedge (d\eta)^q})^{\,\centerdot}
 \eq (q+1)\sum\nolimits_{i}(-1)^i T_i(\dot c)\,\iota_{\,\widehat{\bf T}_i}\omega\wedge (d\eta)^q\\
 \plus d\big(\eta\wedge\dot\eta\wedge (d\eta)^q -(q+1)\,\dot c\,(d\eta)^q\big),\\
 (\iota_{\,\widehat{\bf T}_i}\omega\wedge (d\eta)^q)({\bf T},N,{\bf B})
 \eq\omega({\bf T})\,(d\eta)^q(\widehat{\bf T}_i,N,{\bf B})
 =(d\eta)^q(\widehat{\bf T}_i,N,{\bf B}).
\end{eqnarray*}

(ii) We have  $\dot{\bf X} = \sum\nolimits_{i}(-1)^{i-1} \dot X_i\wedge \widehat{\bf T}_i$,
$\dot X_i=\dot T_i$ and $\dot\omega=0$.
Since $\dot\eta=\iota_{\,\dot{\bf X}}\,d\omega$, see \eqref{E-gv-omegat-dot}, we get
\begin{equation*}
 ({\eta\wedge (d\eta)^q})^{\,\centerdot} = (q+1)\,\iota_{\,\dot{\bf X}}\,d\omega\wedge(d\eta)^{q}
 -d(\eta\wedge(d\eta)^{q-1}\wedge \iota_{\,\dot{\bf X}}\,d\omega).
\end{equation*}
Then
we obtain
\begin{eqnarray*}
 \iota_{\,\dot{\bf X}}\,d\omega\wedge(d\eta)^{q}({\bf T},N,{\bf B}) \eq d\omega(\dot{\bf X},N)\,(d\eta)^{q}({\bf B},{\bf T})\\
 \plus\sum\nolimits_{i}(-1)^{i-1} d\omega(\dot{\bf X},T_i)\,(d\eta)^{q}(\widehat{\bf T}_i,N,{\bf B}) \\
 \plus \sum\nolimits_{k}(-1)^{q+k}d\omega(\dot{\bf X},B_k)\,(d\eta)^{q}({\bf T},N,\widehat{\bf B}_k) ,
\end{eqnarray*}
and \eqref{eq:d_gv2} follows from the above.
\end{proof}

\begin{remark}\rm
If the distribution ${\cal B}$ is integrable then the factor in the right hand side of \eqref{eq:rel_gv1}, see also \eqref{eq:d_gv2}, is
\begin{equation}\label{eq:rel_gv1a}
 (d\eta)^q(\widehat{\bf T}_i,N,{\bf B})=\sum\nolimits_{j}(-1)^{j}\,d\eta(N,B_j)\,(d\eta)^{q-1}(\widehat{\bf T}_i,\widehat{\bf B}_j),
\end{equation}
and if, in addition, the distribution ${\cal D}^\bot$ is integrable then the factor in the right hand side of \eqref{eq:rel_gv1a} is
\begin{equation*}
 (d\eta)^{q-1}(\widehat{\bf T}_i,\widehat{\bf B}_j) {=}
 \sum\nolimits_{\sigma\in S^j_{q-1}} \!\!{\rm sign}(\sigma)\,d\eta((\widehat{\bf T}_i)_1,B_{\sigma_1})\cdot\ldots\cdot
 d\eta((\widehat{\bf T}_i)_{q-1}, B_{\sigma_{q-1}}),
\end{equation*}
where $S^j_{q-1}$ denotes the group of all permutations of the set $\{1,\ldots,\hat j,\ldots,q\}$.
\end{remark}

\section{Variations of metric}

 Let $g=g_0\in{\rm Riem}(M^{2q+1},{\cal D},{\bf T})$ and $g_t\ (|t| < \eps)$ be an arbitrary one-parameter
family of metrics on $(M,{\cal D})$.
Note that the symmetric $(0,2)$-tensor $\dot g$
has only $(q+1)(2q+1)$ independent components.

A family $g_t$ preserving a metric on ${\cal D}$ is called $g^\pitchfork$-\textit{variation}: its tensor $\dot g$ has
$\frac32\,q(q+1)$ nonzero components $\{\dot g_{\,T_i,T_j}=\dot g_{\,T_j,T_i},\dot g_{\,T_i,N},\dot g_{\,T_i,B_k}\}$.
Variations $g_t$, with only $\frac12\,q(q+1)$ nonzero components $\{\dot g_{B_i,B_j}=\dot g_{B_j,B_i}\}$
preserve
both ${\cal D}$ and ${\cal D}^\bot$, and thus produce trivial Euler-Lagrange equations for~\eqref{E-gv-invar-g}.
If ${\cal D}$ is integrable then \eqref{E-gv-invar-g} is constant, hence  Euler-Lagrange equations are trivial.

An arbitrary $g^\pitchfork$-variation of a Riemannian metric $g$ can be decomposed into two cases:

\smallskip
(i)~$g$~varies along ${\cal D}^\bot$ only;\
(ii)~variations preserve $g$ on ${\cal D}$ and $\{T_i\}$ but disturb their orthogonality.

\smallskip
\noindent
Thus, we can divide all nonzero components of $\dot g$ into two sets:
$\{\dot g_{\,T_i,T_j}\}$ and $\{\dot g_{\,T_i,N},\dot g_{\,T_i,B_k}\}$.


\begin{theorem}
Let $T_i\ (1\le i\le q)$ be linear independent vector fields transverse to a codimension $q$ distribution ${\cal D}=\ker\omega$ on $M^{2q+1}\setminus\Sigma$. Then $g\in{\rm Riem}(M,{\cal D},{\bf T})$ is critical for $\gv_{\cal D}$ in \eqref{E-gv-invar-g}
with respect to all variations $g_t$ obeying \eqref{E-omega-T-t} and
\begin{subequations}
\begin{eqnarray}\label{E-cond2-Sigma}
\nonumber
 &&
 \int_M \|
 (\tr\dot B)\sum\nolimits_{i}(-1)^{i}(d\eta)^q(\widehat{\bf T}_i,N,{\bf B})\cdot T_i\\
 && +\,\dot\eta\wedge\eta\wedge(d\eta)^{q-1} \!+(q+1)(\tr\dot C)\,(d\eta)^q \|^p\,{\rm d} V_g <\infty ,\\
\label{E-cond2-Sigma-b}
 &&
 \int_M \|\iota_{\,\dot{\bf X}}\,d\omega\wedge\eta\wedge(d\eta)^{q-1} \|^p\,{\rm d} V_g <\infty
\end{eqnarray}
\end{subequations}
for some $p$ such that $(k-1)(p-1)\ge1$, if and only the following $q^2+q+1$ equations hold on $U$:
\begin{subequations}
\begin{eqnarray}
\label{E-EL-1}
 &&\hskip-8mm \Div\big(\sum\nolimits_{i}(-1)^{i}(d\eta)^q(\widehat{\bf T}_i,N,{\bf B})\cdot T_i\big)= 0,\\
\label{E-EL-2}
 &&\hskip-8mm (d\eta)^q(\widehat{\bf T}_i,N,{\bf B})\cdot\|H^\bot\|
 =-\sum\nolimits_{k}(-1)^{k} (d\eta)^q({\bf T},N,\widehat{\bf B}_k)\,d\omega(N,\widehat{\bf T}_i,B_k), \\
\label{E-EL-3}
 &&\hskip-9mm (d\eta)^q({\bf B},{\bf T})\,d\omega(B_j,\widehat{\bf T}_i,N)
 {=}-\!\!\sum\limits_{k}\!(-1)^{q+k} (d\eta)^q({\bf T},N, \widehat{\bf B}_k)d\omega(B_j,\widehat{\bf T}_i,B_k).
\end{eqnarray}
\end{subequations}
For integrable ${\cal D}$, equations $($\ref{E-EL-1}-c$)$ reduce to the expected trivial equalities.
\end{theorem}

\begin{proof}
According to Lemma~\ref{C-02}, one should consider only two cases.

Case 1.
Let $\dot T_i=\dot C^j_i\,T_j$ and $\dot\omega=-\dot c\,\omega$, where $ c(t)=\det C(t)$.
Differentiating $g_t(T_i(t),T_j(t))=\delta_{ij}$ at $t=0$ we obtain
$\dot g_{ij}=-\dot C^j_i-\dot C^i_j$. Hence,
 $\sum\nolimits_i \dot g_{\,T_i,T_i} = -2\tr\dot C$.
 By Lemma~\ref{C-02}(i), and using
\[
  \Div(Q_i\tr\dot C\cdot T_i)=(\tr\dot C)\Div(Q_i\cdot T_i)+Q _iT_i(\tr\dot C)
\]
with $ Q_i = (-1)^{i}(d\eta)^q(\widehat{\bf T}_i,N,{\bf B})\,{\rm d}V_g $,
we have
\begin{eqnarray*}
 && ({\eta\wedge (d\eta)^q})^{\,\centerdot}  =
 -(q+1)(\tr\dot C)\Div\big(\sum\nolimits_{i}(-1)^{i}(d\eta)^q(\widehat{\bf T}_i,N,{\bf B})\cdot T_i\big)\,{\rm d}V_g \\
 && +\,\Div((\tr\dot C)\sum\nolimits_{i}(-1)^{i}(d\eta)^q(\widehat{\bf T}_i,N,{\bf B})\cdot T_i)\,{\rm d}V_g\\
 && +d\big(\dot\eta\wedge\eta\wedge(d\eta)^{q-1} +(q+1)\,(\tr\dot C)\,(d\eta)^q\big) .
\end{eqnarray*}
By Stokes theorem and \eqref{E-cond2-Sigma}, the Euler-Lagrange equations read as
\eqref{E-EL-1}.

Case 2.
Now, let $T_i(t)=T_i+X_i(t)$ be the orthonormal frame of ${\cal D}_t^\bot$ with respect to $g_t$
for some vector fields $X_i(t)\in\mathfrak{X}_{\cal D}$ with $X_i(0)=0$.
Differentiating $g_t(T_i+X_i(t),N)=0$ at $t=0$ we obtain
\[
 \<\dot X_i,N\> =-\dot g_{\,T_i,N}.
\]
Similarly,
$\<\dot X_i,B_j\>=-\dot g_{\,T_i,B_j}$. Hence,
 $\dot X_i = -\dot g_{\,T_i,N} N -\sum\nolimits_j \dot g_{\,T_i,B_j} B_j$,
and
\begin{eqnarray*}
 d\omega(\dot{\bf X},T_i) \eq \sum\nolimits_{j}(-1)^{j}\big[d\omega(N,\widehat{\bf X}_j,T_i)\,\dot g_{\,T_j,N}
 +\sum\nolimits_{k} d\omega(B_k,\widehat{\bf X}_j,T_i)\,\dot g_{\,T_i,B_k}\big]\\
 \eq (-1)^{i}\big[d\omega(N,\widehat{\bf T}_i,T_i)\,\dot g_{\,T_i,N}
 +\sum\nolimits_{k} d\omega(B_k,\widehat{\bf T}_i,T_i)\,\dot g_{\,T_i,B_k}\big]\\
 \eq (-1)^{q}\big[d\omega(N,{\bf T})\,\dot g_{\,T_i,N}
 +\sum\nolimits_{k} d\omega(B_k,{\bf T})\,\dot g_{\,T_i,B_k}\big]\\
 \eq (-1)^{q-1}\,\|H^\bot\|\,\dot g_{\,T_i,N},\\
 d\omega(\dot{\bf X},N) \eq -\sum\nolimits_{k,i}(-1)^{i-1} d\omega(B_k,\widehat{\bf T}_i,N)\,\dot g_{\,T_i,B_k} ,\\
 d\omega(\dot{\bf X},B_k) \eq -\!\sum\nolimits_{i}(-1)^{i-1}\big( d\omega(N,\widehat{\bf T}_i, B_k)\,\dot g_{\,T_i,N}
 {+}\!\sum\nolimits_{j}d\omega(B_j,\widehat{\bf T}_i,B_k)\,\dot g_{\,T_i,B_j} \big).
\end{eqnarray*}
Here we used $\dot{\bf X} = \sum\nolimits_{j}(-1)^{j-1} \dot X_i\wedge \widehat{\bf T}_i$ and $d\omega(N,{\bf T})=-\|H^\bot\|$.
Then, by Lemma~\ref{C-02}(ii), we have
\begin{eqnarray*}
 &&\hskip3mm ({\eta\wedge (d\eta)^q})^{\,\centerdot}  = (q+1)\big(
 -\sum\nolimits_{i,j}(-1)^{i-1}(d\eta)^q({\bf B},{\bf T})\,d\omega(B_j,\widehat{\bf T}_i,N)\,\dot g_{\,T_i,B_j} \\
 && +\sum\nolimits_{i}(-1)^{q+i-2}\|H^\bot\|\,(d\eta)^q(\widehat{\bf T}_i,N,{\bf B})\,\dot g_{\,T_i,N} \\
 && -\sum\nolimits_{k}(-1)^{q+k+i-1}\big( d\omega(N,\widehat{\bf T}_i,B_k)\,\dot g_{\,T_i,N} \\
 && +\sum\nolimits_{j} d\omega(B_j,\widehat{\bf T}_i,B_k)\,\dot g_{\,T_i,B_j}\big)(d\eta)^q({\bf T},N,\widehat{\bf B}_k)\big)\,{\rm d}V_g
 -d(\eta\wedge(d\eta)^{q-1}\!\wedge\iota_{\,\dot{\bf X}}\,d\omega).
\end{eqnarray*}
By Stokes theorem and \eqref{E-cond2-Sigma-b}, the
vanishing of $\dot g_{\,T_i,N},\dot g_{\,T_i,B_k}$ components provides
(\ref{E-EL-2},c).
\end{proof}

\begin{remark}\rm
For $q=1$, (\ref{E-EL-1}-c) reduce to the following system of equations on $U$, see \cite[Theorem~4.2]{rw-gv1}:
\begin{subequations}
\begin{eqnarray}\label{E-EL1-1}
 \Div( \Div( {\mathcal T}_{N,B}\cdot T)\cdot T)
  \eq 0,\\
\label{E-EL1-2}
  \Div( {\mathcal T}_{N,B}\cdot T) - (T(\log k) - h_{N,N}){\mathcal T}_{N,B} \eq  0,\\
\label{E-EL1-3}
  (\tau - h_{B,N})\,{\mathcal T}_{N,B} \eq  0 .
\end{eqnarray}
\end{subequations}
\end{remark}


\begin{corollary}
Let $g\in{\rm Riem}(M,{\cal D},{\bf T})$ and either $(d\eta)^q=0$ or the normal distribution ${\cal D}^\bot$ be harmonic. 
Then $g$ is a critical point for $\gv_{\cal D}$
with respect to all variations of metric obeying \eqref{E-omega-T-t} and {\rm(\ref{E-cond2-Sigma},b)}.
\end{corollary}

\begin{proof}
If $(d\eta)^q=0$ then  (\ref{E-EL-1}-c) hold, hence $g$ is critical.
If $H^\bot=0$, then $\eta=0$, $d\eta = 0$, see \eqref{eq1:pw2013}, and (\ref{E-EL-1}-c) and (\ref{E-cond2-Sigma},b) are satisfied trivially.
\end{proof}

\end{document}